\documentclass[12pt,thmsa,a4paper]{article}
\usepackage{amssymb}
\usepackage{amsfonts}
\usepackage{amsmath}
\usepackage[colorlinks=true]{hyperref}
\usepackage{graphicx}
\usepackage{float}

\allowdisplaybreaks
\hypersetup{urlcolor=blue, citecolor=blue}
\newtheorem{theorem}{Theorem}

\newtheorem{example}{Example}

\newtheorem{lemma}{Lemma}

\newtheorem{remark}{Remark}

\newenvironment{proof}[1][Proof]{\noindent\textbf{#1.} }{\ \rule{0.5em}{0.5em}}

\begin{document}

\title{Bifurcation Curve Diagrams for a Diffusive Generalized Logistic
Problem with Minkowski Curvature Operator and Constant-Yield Harvesting%
\thanks{%
2020 Mathematics Subject Classification: 34B15, 34B18, 34C23, 74G35.}\thanks{%
Keywords: positive solution, Minkowski curvature problem, bifurcation curve}}
\author{Shao-Yuan Huang\thanks{%
Department of Mathematics and Information Education, National Taipei
University of Education, Taipei 106, Taiwan. \ E-mail address:\textit{\ }%
syhuang@mail.ntue.edu.tw}}
\date{}
\maketitle

\begin{abstract}
This paper investigates the bifurcation diagrams of positive solutions for a
one-dimensional diffusive generalized logistic boundary-value problem with
Minkowski curvature operator and constant yield harvesting
\begin{equation*}
\left\{
\begin{array}{l}
-\left( u^{\prime }/\sqrt{1-{u^{\prime }}^{2}}\right) ^{\prime }=\lambda
g(u)-\mu ,\text{ \ in }\left( -L,L\right) , \\
u(-L)=u(L)=0,%
\end{array}%
\right.
\end{equation*}%
where $\lambda ,L,\mu >0$, $g(0)=0$ and there exists $\sigma >0$ such that $%
g\in C[0,\sigma ]\cap C^{2}(0,\sigma )$, $\left( \sigma -u\right) g(u)>0$
for $u\neq \sigma $, and $g^{\prime \prime }(u)<0$ on $\left( 0,\sigma
\right) $. We prove that the corresponding bifurcation curves on both the $%
(\lambda ,\left\Vert u\right\Vert _{\infty })$-plane and $(\mu ,\left\Vert
u\right\Vert _{\infty })$-plane are $\subset $-shaped. Furthermore, by
characterizing the bifurcation set on the $(\mu ,\lambda )$-plane, we
determine the exact multiplicity of positive solutions. As an application to
population and ecological models, we further consider the nonlinearity $g(u)=u^{p}\left( 1-\left( u/K\right) ^{q}\right) ^{r}$, where $p,q,r,K>0$ satisfy conditions such that $g$ is concave on $\left(0,K\right)$.
\end{abstract}

\section{Introduction}

In this paper, we study the bifurcation diagrams for a one-dimensional
diffusive generalized logistic problem with the Minkowski curvature operator
and constant-yield harvesting:%
\begin{equation}
\left\{
\begin{array}{l}
-\left( u^{\prime }/\sqrt{1-{u^{\prime }}^{2}}\right) ^{\prime }=\lambda
g(u)-\mu ,\text{ \ in }\left( -L,L\right) , \\
u(-L)=u(L)=0,%
\end{array}%
\right.  \label{eq1}
\end{equation}%
where $\lambda ,L,\mu >0$, $g(0)=0$ and there exists $\sigma >0$ such that $%
\left( \sigma -u\right) g(u)>0$ for $u\neq \sigma $, $g\in C[0,\sigma ]\cap
C^{2}(0,\sigma )$ and $g^{\prime \prime }(u)<0$ on $\left( 0,\sigma \right) $%
. Clearly, $g^{\prime }(0^{+})\in (0,\infty ]$ and there exists $u_{0}\in
\left( 0,\sigma \right) $ such that
\begin{equation*}
g^{\prime }(u_{0})=0,\text{ \ }g^{\prime }(u)>0\text{ on }\left(
0,u_{0}\right) \text{, \ and \ }g^{\prime }(u)<0\text{ on }(u_{0},\sigma ).
\end{equation*}

Obviously, (\ref{eq1}) is a semipositone problem. Semipositone problems are
not only of mathematical interest but also have practical applications in
various fields, including the buckling of mechanical systems, the design of
suspension bridges, chemical reactions, and population models with
harvesting effort, cf. \cite{Ali, Girao, Jiang, Myerscough, Oruganti,
Oruganti2} and references therein. While (\ref{eq1}) may admit nonnegative
solutions, this paper focuses exclusively on the positive solutions of (\ref%
{eq1}).

We now define the bifurcation curves $S_{\mu }$ and $\Sigma _{\lambda }$ of
positive solutions of (\ref{eq1}) as follows:

\begin{itemize}
\item[(i)] For $\mu >0$, the bifurcation curve $S_{\mu }$ of positive
solutions of (\ref{eq1}) is defined on the\ $(\lambda ,\left\Vert
u\right\Vert _{\infty })$-plane by%
\begin{equation}
S_{\mu }\equiv \left\{ \left( \lambda ,\left\Vert u_{\lambda }\right\Vert
_{\infty }\right) :\lambda >0\text{ and }u_{\lambda }\text{ is a positive
solution of (\ref{eq1})}\right\} .  \label{SL}
\end{equation}

\item[(ii)] For $\lambda >0$, the bifurcation curve $\Sigma _{\lambda }$ of
positive solutions of (\ref{eq1}) is defined on the $(\mu ,\left\Vert
u\right\Vert _{\infty })$-plane by
\begin{equation}
\Sigma _{\lambda }\equiv \left\{ \left( \mu ,\left\Vert u_{\mu }\right\Vert
_{\infty }\right) :\mu >0\text{ and }u_{\mu }\text{ is a positive solution
of (\ref{eq1})}\right\} .  \label{CL}
\end{equation}
\end{itemize}

\noindent It is well known that studying the exact shape of the bifurcation
curves $S_{\mu }$ or $\Sigma _{\lambda }$ of (\ref{eq1}) is equivalent to
studying the exact multiplicity of positive solutions of (\ref{eq1}).
Therefore, many researchers have devoted significant efforts to studying the
shapes of bifurcation curves, cf. \cite{Huang1,Huang2,Huang3,Huang4} and
references therein.

As $\mu >0$, (\ref{eq1}) considers both the intrinsic growth of the species
and the effect of external harvesting. The nonlinear term $\lambda g(u)$
characterizes the generalized logistic growth, where $\lambda $ represents
the intrinsic growth rate, and $\mu $ denotes a constant harvesting rate,
illustrating the impact of external harvesting on the population dynamics.
This model has significant applications in resource management, particularly
in fisheries, where determining the optimal balance between harvesting and
sustainable growth is essential.

Recently, Hung et al. \cite{Hung, Hung2} studied the following closely
related model%
\begin{equation}
\left\{
\begin{array}{l}
-u^{\prime \prime }=\lambda g(u)-\mu ,\text{ \ in }\left( -1,1\right) , \\
u(-1)=u(1)=0.%
\end{array}%
\right.  \label{eq3}
\end{equation}%
Similarly, by analyzing the shape of the corresponding bifurcation curve of (%
\ref{eq3}), one can determine the exact number of positive solutions of (\ref%
{eq3}). To achieve this, Hung et al. \cite{Hung} initially explored the
properties of the nonlinearity $\lambda g(u)-\mu $, yielding several key
results.

\begin{lemma}[{\protect\cite[(1.4)--(1.6)]{Hung}}]
\label{LH}Assume that
\begin{equation*}
\lambda >\frac{\mu }{\max\limits_{u\in \lbrack 0,\sigma ]}g(u)}=\frac{\mu }{%
g(u_{0})}\equiv \lambda _{\min }.
\end{equation*}%
Let $G(u)\equiv \int_{0}^{u}g(t)dt$, $f_{\mu ,\lambda }(u)\equiv \lambda
g(u)-\mu $ and $F_{\mu ,\lambda }(u)\equiv \lambda G(u)-\mu u$. Then the
following statements (i)--(iii) hold:

\begin{itemize}
\item[(i)] There exist $\varsigma _{\mu ,\lambda },\beta _{\mu ,\lambda }\in
(0,\sigma )$ such that%
\begin{equation*}
f_{\mu ,\lambda }(u)\left\{
\begin{array}{ll}
<0 & \text{on }(0,\varsigma _{\mu ,\lambda })\cup (\beta _{\mu ,\lambda
},\sigma ), \\
=0 & \text{for }u=\varsigma _{\mu ,\lambda }\text{ and }u=\beta _{\mu
,\lambda }, \\
>0 & \text{on }(\varsigma _{\mu ,\lambda },\beta _{\mu ,\lambda }).%
\end{array}%
\right.
\end{equation*}

\item[(ii)] There exists a unique $c^{\ast }\in (u_{0},\sigma )$ such that%
\begin{equation*}
\left[ \frac{G(u)}{u}\right] ^{\prime }\left\{
\begin{array}{ll}
>0 & \text{for }0<u<c^{\ast }, \\
=0 & \text{for }u=c^{\ast }, \\
<0 & \text{for }c^{\ast }<u<\sigma ,%
\end{array}%
\right. \text{ \ and \ }\frac{G(c^{\ast })}{c^{\ast }}=g(c^{\ast }).
\end{equation*}

\item[(iii)] For
\begin{equation}
\lambda \geq \lambda _{\mu }\equiv \frac{\mu }{g(c^{\ast })},  \label{w0}
\end{equation}%
there exists a unique $\theta _{\mu ,\lambda }\in \left( \varsigma _{\mu
,\lambda },\beta _{\mu ,\lambda }\right) $ such that $F_{\mu ,\lambda
}(\theta _{\mu ,\lambda })=0$. Furthermore, $\theta _{\mu ,\lambda _{\mu
}}=\beta _{\mu ,\lambda _{\mu }}$.
\end{itemize}
\end{lemma}

Hung et al. \cite{Hung} proved that, under varying conditions, the
corresponding bifurcation curves of (\ref{eq3}) are $\subset $-shaped on
both the $(\lambda ,\left\Vert u\right\Vert _{\infty })$-plane and $(\mu
,\left\Vert u\right\Vert _{\infty })$-plane. They further studied the
bifurcation surface in the $\left( \mu ,\lambda ,\left\Vert u\right\Vert
_{\infty }\right) $-space, and determined the bifurcation set on the $\left(
\mu ,\lambda \right) $-plane to obtain the exact multiplicity of positive
solutions of (\ref{eq3}). Motivated by their results, we extend this
analysis to the Minkowski curvature problem.

There are some references on bifurcation surfaces and bifurcation sets, cf.
\cite{Huang4,Hung,Hung3}, which facilitate the study of how the number of
positive solutions changes with respect to parameter variations. However, to
the best of my knowledge, there are no references to study such issues for
Minkowski curvature problem. To fill this gap, we study the bifurcation set
for Minkowski curvature problem (\ref{eq1}).

Finally, we present an example of a generalized logistic problem%
\begin{equation}
\left\{
\begin{array}{l}
-\left( u^{\prime }/\sqrt{1-{u^{\prime }}^{2}}\right) ^{\prime }=\lambda
u^{p}\left[ 1-\left( \frac{u}{K}\right) ^{q}\right] ^{r}-\mu ,\text{ \ in }%
\left( -L,L\right) , \\
u(-L)=u(L)=0,%
\end{array}%
\right.   \label{eq2}
\end{equation}%
where $K>0$ and $p,q,r>0$ satisfy one of the following conditions (h$_{1}$%
)--(h$_{4}$) holds:

\begin{itemize}
\item[(h$_{1}$)] $p+qr-1=0$ and $0<r\leq 1$.

\item[(h$_{2}$)] $p-1\leq 0<p+qr-1$ and $0<r\leq 1$.

\item[(h$_{3}$)] $p+qr-1<0\leq 2p-q+2qr-1$ and $0<r\leq 1$.

\item[(h$_{4}$)] $p+qr-1<0$, $2p-q+2qr-1<0$ and $0<r<\frac{4p\left(
1-p\right) }{\left( q-1\right) ^{2}+4pq}$.
\end{itemize}

\noindent Note that the nonlinearity in (\ref{eq2}) is concave if and onl if
one of (h$_{1}$)--(h$_{4}$) holds. Problem (\ref{eq2}) can be viewed as a
steady-state diffusive single-species model with generalized logistic
growth, where $K$ is the carrying capacity and $p,$ $q$ describe the
low-/high-density growth profiles. The additional constant term $-\mu $
represents constant-yield harvesting (quota harvesting), which is known to
introduce semipositone effects and threshold phenomena, including possible
nonexistence and multiplicity of positive equilibria depending on $\left(
\mu ,\lambda \right) .$Such harvesting mechanisms and their
parameter-dependent bifurcation behavior have been extensively investigated
in logistic-type reaction--diffusion models, cf. \cite{Ali, Girao, Hung,
Hung2}. Our results provide an explicit $\left( \mu ,\lambda \right) $%
-parameter classification for the existence and multiplicity of positive
solutions in the Minkowski-curvature setting.

As $\mu =0$ (i.e., no harvesting), problem (\ref{eq1}) reduces to the
classical diffusive logistic equation, commonly expressed as $g(u)=u\left(
1-u\right) $. This model has been widely studied, with foundational results
provided in \cite{Huang2, H1} and references therein. In this case, the
equation only considers the intrinsic growth of the species, without any
external harvesting influence. The nonlinear term $\lambda g(u)$ ensures
positivity, making it easier to analyze the bifurcation structure. Since $%
g^{\prime \prime }(u)<0$ on $\left( 0,\sigma \right) $, it is easy to prove
that%
\begin{equation*}
\left( \frac{g(u)}{u}\right) ^{\prime }<0\text{ \ for }0<u<\sigma .
\end{equation*}%
Consequently, by \cite[Theorems 2.1 and 2.2]{Huang2}, the following theorem
is established.

\begin{theorem}
\label{RT}Consider (\ref{eq1}) with $\mu =0$. Let%
\begin{equation}
\kappa \equiv \left\{
\begin{array}{ll}
\frac{\pi ^{2}}{4g^{\prime }(0^{+})L^{2}} & \text{if }g^{\prime }(0^{+})\in
(0,\infty ), \\
0 & \text{if }g^{\prime }(0^{+})=\infty ,%
\end{array}%
\right. \text{ \ and \ }m_{\sigma ,L}\equiv \min \{\sigma ,L\}.  \label{K}
\end{equation}%
Then the corresponding bifurcation curve is monotone increasing, starts from
$(\kappa ,0)$ and goes to $(\infty ,m_{\sigma ,L})$.
\end{theorem}

The paper is organized as follows. Section 2 presents the main results and
an example. Section 3 provides several lemmas necessary for proving the main
results, while Section 4 contains the proofs of the main results. Section 5
provides the proof of Lemma \ref{Le10}.

\section{Main Results}

In this section, we present our main results. In Theorems \ref{T1} and \ref%
{T3}, we respectively establish the shapes of the bifurcation curves $S_{\mu
}$ and $\Sigma _{\lambda }$. Subsequently, we introduce the concepts of the
bifurcation surface and bifurcation set. For further details on the study of
bifurcation surfaces and bifurcation sets, readers may refer to references
\cite{Huang4, Hung, Hung3}. Moreover, in Theorem \ref{T2}, we provide the
exact multiplicity of positive solutions of (\ref{eq1}).

Recall the numbers $c^{\ast }$, $\theta _{\mu ,\lambda }$, $\kappa $ and $%
m_{\sigma ,L}$ defined by Lemma \ref{LH} and (\ref{K}), respectively. Let $%
c_{L}^{\ast }\equiv \min \{c^{\ast },L\}$.

\begin{theorem}
\label{T1}Consider (\ref{eq1}) with varying $\mu >0$. Then there exists $%
\bar{\lambda}\in (\lambda _{\mu },\infty )$ such that the bifurcation curve $%
S_{\mu }$ is continuous, starts from $(\bar{\lambda},\left\Vert u_{\bar{%
\lambda}}\right\Vert _{\infty })=(\bar{\lambda},\theta _{\mu ,\bar{\lambda}%
}) $, goes to $(\infty ,m_{\sigma ,L})$ and is $\subset $-shaped with
exactly one turning point $\left( \lambda ^{\ast },\left\Vert u_{\lambda
^{\ast }}\right\Vert _{\infty }\right) $ on the $\left( \lambda ,\left\Vert
u_{\lambda }\right\Vert _{\infty }\right) $-plane, see Figure \ref{fig1}.
Furthermore,

\begin{itemize}
\item[(i)] $\bar{\lambda}$ and $\left\Vert u_{\bar{\lambda}}\right\Vert
_{\infty }$ are strictly increasing and continuous functions with respect to
$\mu >0,$%
\begin{equation*}
\lim_{\mu \rightarrow 0^{+}}\left( \bar{\lambda},\left\Vert u_{\bar{\lambda}%
}\right\Vert _{\infty }\right) =\left( 4\kappa ,0\right) \text{ \ and \ }%
\lim_{\mu \rightarrow \infty }\left( \bar{\lambda},\left\Vert u_{\bar{\lambda%
}}\right\Vert _{\infty }\right) =\left( \infty ,c_{L}^{\ast }\right) .
\end{equation*}

\item[(ii)] $\lambda ^{\ast }$ and $\left\Vert u_{\lambda ^{\ast
}}\right\Vert _{\infty }$ are strictly increasing and continuous functions
with respect to $\mu >0,$%
\begin{equation*}
\lim_{\mu \rightarrow 0^{+}}\left( \lambda ^{\ast },\left\Vert u_{\lambda
^{\ast }}\right\Vert _{\infty }\right) =\left( 4\kappa ,0\right) \text{ \
and \ }\lim_{\mu \rightarrow \infty }\left( \lambda ^{\ast },\left\Vert
u_{\lambda ^{\ast }}\right\Vert _{\infty }\right) =\left( \infty
,c_{L}^{\ast }\right) .
\end{equation*}
\end{itemize}
\end{theorem}

\begin{figure}[h]
\centering
\includegraphics[width=6.1in,height=2.6in,keepaspectratio]{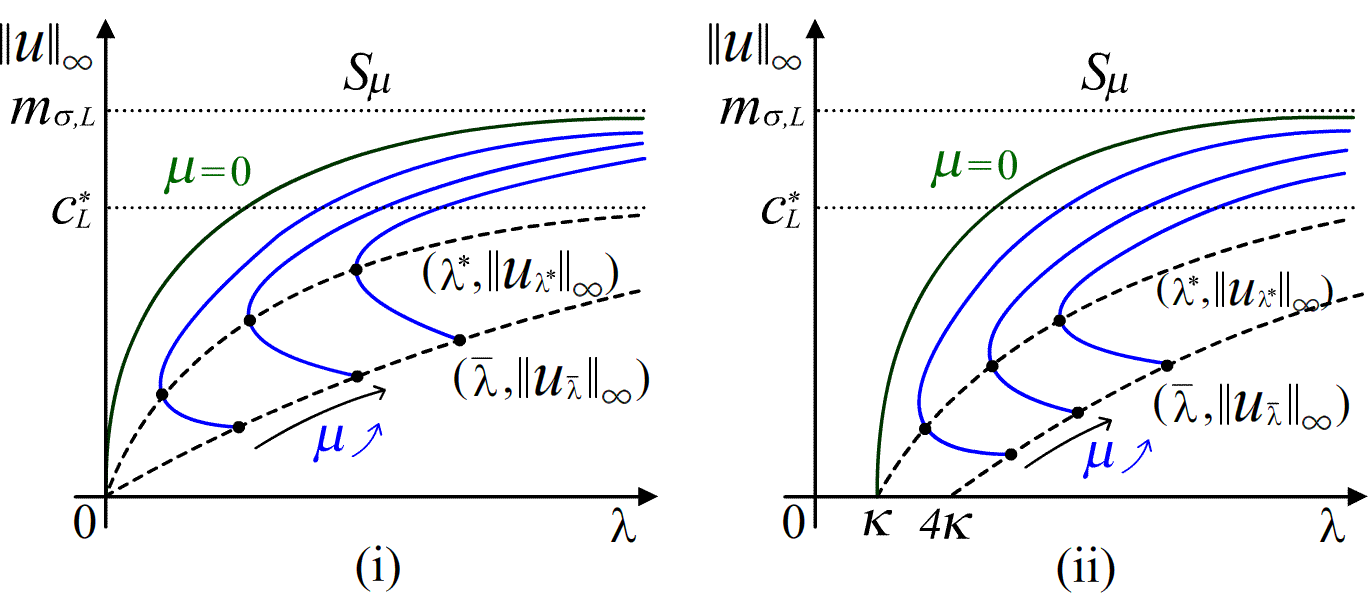}
\caption{Graphs of $S_{\protect\mu }$. $S_{\protect\mu }$ is monotone
increasing for $\protect\mu =0$, and $\subset $-shaped for $\protect\mu >0$.
(i) $g^{\prime }(0^{+})=\infty $. (ii) $g^{\prime }(0^{+})\in (0,\infty )$.}
\label{fig1}
\end{figure}
Let%
\begin{equation*}
T_{0,\lambda }(\alpha )\equiv \int_{0}^{\alpha }\frac{\lambda \left[
G(\alpha )-G(u)\right] +1}{\sqrt{\lambda ^{2}\left[ G(\alpha )-G(u)\right]
^{2}+2\lambda \left[ G(\alpha )-G(u)\right] }}du\text{ \ for }0<\alpha
<\sigma \text{ and }\lambda >0,
\end{equation*}%
where $G$ is defined in Lemma \ref{LH}. Notice that $T_{0,\lambda }(\alpha )$
is a time-map for (\ref{eq1}) with $\mu =0$, cf. \cite{Hung}. Let%
\begin{equation}
\eta \equiv \left\{
\begin{array}{ll}
\frac{\pi }{2\sqrt{\lambda g^{\prime }(0^{+})}} & \text{if }g^{\prime
}(0^{+})\in (0,\infty ), \\
0 & \text{if }g^{\prime }(0^{+})=\infty .%
\end{array}%
\right.   \label{E}
\end{equation}%
If $L>\eta $, by Lemma \ref{Le0} stated below, there exists unique $\gamma
_{\lambda }\in \left( 0,\sigma \right) $ such that
\begin{equation}
T_{0,\lambda }(\gamma _{\lambda })=L.  \label{r}
\end{equation}%
Therefore, we have the following Theorem \ref{T3}.

\begin{theorem}
\label{T3}Consider (\ref{eq1}) with varying $\lambda >0$. Then the following
statements (i)--(ii) hold:

\begin{itemize}
\item[(i)] Assume that $g^{\prime }(0^{+})\in (0,\infty )$.

\begin{itemize}
\item[(a)] If $0<\lambda \leq \kappa $, then the bifurcation curve $\Sigma
_{\lambda }$ does not exist.

\item[(b)] If $\kappa <\lambda \leq 4\kappa $, then the bifurcation curve $%
\Sigma _{\lambda }$ is continuous, starts from $(0,0)$, goes to $(0,\gamma
_{\lambda })$, and is reversed $\subset $-shaped on the $\left( \mu
,\left\Vert u_{\mu }\right\Vert _{\infty }\right) $-plane, see Figure \ref%
{fig2}(i).
\end{itemize}

\item[(ii)] Assume that $g^{\prime }(0^{+})\in (0,\infty ]$ and $\lambda
>4\kappa $. Then there exists $\bar{\mu}\in (0,\mu _{\lambda })$ such that
the bifurcation curve $\Sigma _{\lambda }$ is continuous, starts from $(\bar{%
\mu},\left\Vert u_{\bar{\mu}}\right\Vert _{\infty })=(\bar{\mu},\theta _{%
\bar{\mu},\lambda })$, goes to $(0,\gamma _{\lambda })$, and is reversed $%
\subset $-shaped on the $\left( \mu ,\left\Vert u_{\mu }\right\Vert _{\infty
}\right) $-plane, see Figure \ref{fig2}(ii). Furthermore, $\bar{\mu}$ and $%
\left\Vert u_{\bar{\mu}}\right\Vert _{\infty }$ are strictly increasing and
continuous functions with respect to $\lambda >0$,%
\begin{equation}
\lim_{\lambda \rightarrow \left( 4\kappa \right) ^{+}}\left( \bar{\mu}%
,\left\Vert u_{\bar{\mu}}\right\Vert _{\infty }\right) =\left( 0,0\right)
\text{ \ and \ }\lim_{\lambda \rightarrow \infty }\left( \bar{\mu}%
,\left\Vert u_{\bar{\mu}}\right\Vert _{\infty }\right) =\left( \infty
,c_{L}^{\ast }\right) .  \label{T3A}
\end{equation}
\end{itemize}
\end{theorem}

\begin{remark}
Theorem \ref{T3}(ii) also covers the case when $g^{\prime }(0^{+})\in
(0,\infty )$ and $\lambda >4\kappa $, which is the seemingly omitted case in
Theorem \ref{T3}(i). In addition, as $g^{\prime }(0^{+})=\infty $, the
condition $\lambda >4\kappa $ reduces to $\lambda >0$.
\end{remark}

\begin{figure}[h]
\centering
\includegraphics[width=6.1in,height=2.6in,keepaspectratio]{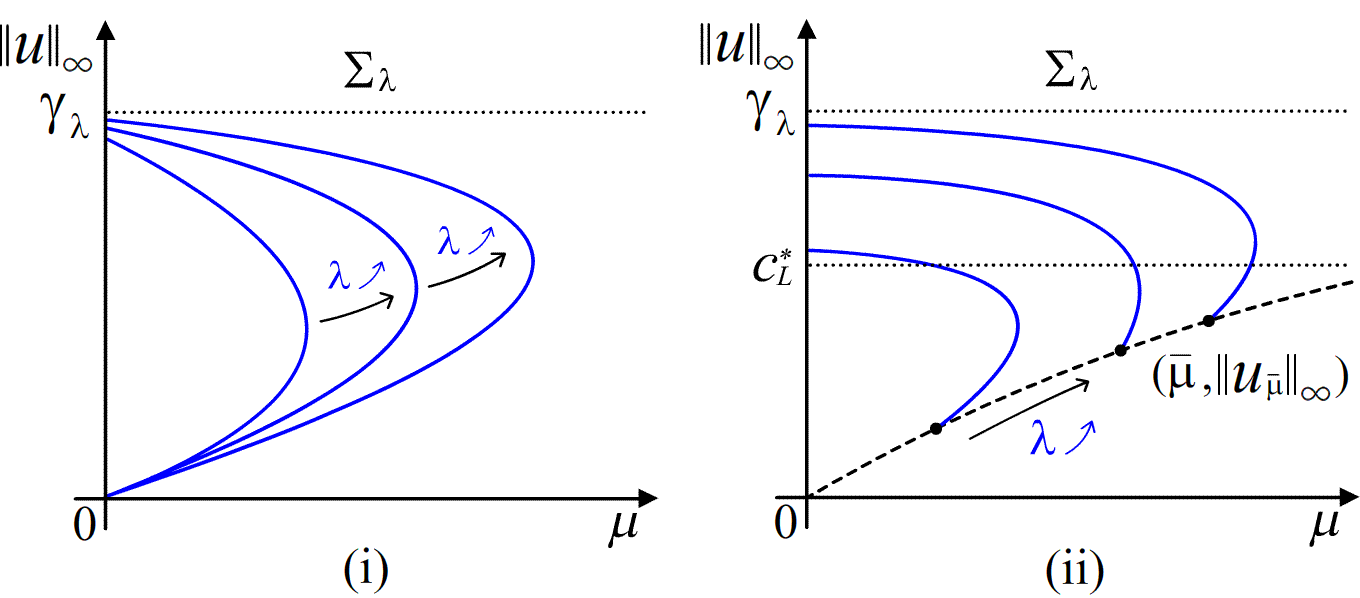}
\caption{Graphs of $\Sigma _{\protect\lambda }$. $\Sigma _{\protect\lambda }$
is reversed $\subset $-shaped for $\protect\lambda >\protect\kappa $. (i) $%
\protect\kappa <\protect\lambda \leq 2\protect\kappa .$ (ii) $\protect%
\lambda >2\protect\kappa $.}
\label{fig2}
\end{figure}
In the $\left( \mu ,\lambda ,\left\Vert u\right\Vert _{\infty }\right) $%
-space, the \emph{bifurcation surface} $\Gamma $ of (\ref{eq1}) is defined by%
\begin{equation*}
\Gamma \equiv \left\{ \left( \mu ,\lambda ,\left\Vert u\right\Vert _{\infty
}\right) :\left( \mu ,\lambda \right) \in \Omega \text{ \ and \ }u_{\mu
,\lambda }\text{ is a positive solution of (\ref{eq1})}\right\} .
\end{equation*}%
cf. \cite{Huang4, Hung, Hung3}. Recall that, by Theorem \ref{T1}, for fixed $%
\mu
>0$, $S_{\mu }$ is continuous, starts from $(\bar{\lambda},\left\Vert u_{%
\bar{\lambda}}\right\Vert _{\infty })$ and is $\subset $-shaped with exactly
one turning point $\left( \lambda ^{\ast },\left\Vert u_{\lambda ^{\ast
}}\right\Vert _{\infty }\right) $. So the bifurcation surface has the
appearance of a surface with the curve%
\begin{equation*}
C_{1}\equiv \left\{ (\mu ,\bar{\lambda}(\mu ),\left\Vert u_{\mu ,\bar{\lambda%
}(\mu )}\right\Vert _{\infty }):\mu >0\right\}
\end{equation*}%
being the set of all starting points $(\mu ,\bar{\lambda}(\mu ),\left\Vert
u_{\mu ,\bar{\lambda}(\mu )}\right\Vert _{\infty })=(\mu ,\bar{\lambda}%
,\theta _{\mu ,\bar{\lambda}})$, and with the curve%
\begin{equation*}
C_{2}\equiv \left\{ (\mu ,\lambda ^{\ast }(\mu ),\left\Vert u_{\mu ,\lambda
^{\ast }(\mu )}\right\Vert _{\infty }):\mu >0\right\}
\end{equation*}%
being the fold curve of $\Gamma $. We define the \emph{bifurcation set} $%
B_{\Gamma }\equiv B_{1}\cup B_{2}$ where%
\begin{equation*}
B_{1}\equiv \left\{ \left( \mu ,\bar{\lambda}(\mu )\right) :\mu >0\right\}
\text{ \ and \ }B_{2}\equiv \left\{ \left( \mu ,\lambda ^{\ast }(\mu
)\right) :\mu >0\right\} .
\end{equation*}%
Clearly, $B_{1}$ and $B_{2}$ are the projection of the curves $C_{1}$ and $%
C_{2}$ on the $\left( \mu ,\lambda \right) $-parameter plane, respectively,
see Figure \ref{fig5}.

\begin{figure}[h]
\centering
\includegraphics[width=6.1in,height=2.6in,keepaspectratio]{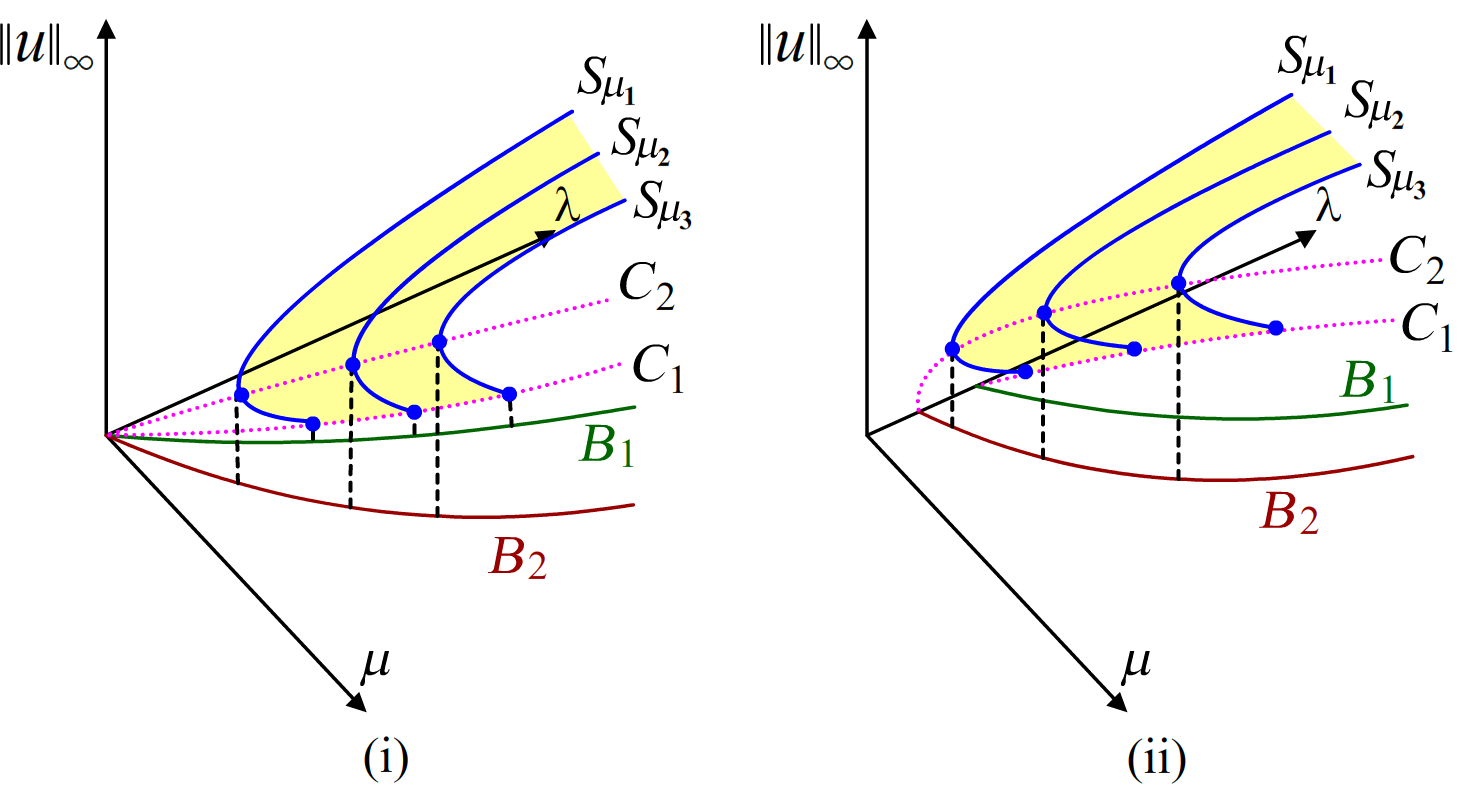}
\caption{The bifurcation set $B_{\Gamma }$. (i) $g^{\prime }(0^{+})=\infty $%
. (ii) $g^{\prime }(0^{+})\in (0,\infty )$.}
\label{fig5}
\end{figure}
In the following theorem, we examine the structure of the bifurcation set,
and exact multiplicity of positive solutions of (\ref{eq1}).

\begin{theorem}[See Figure \protect\ref{fig3}]
\label{T2}Consider (\ref{eq1}). Then the following statements (i)--(ii) hold.

\begin{itemize}
\item[(i)] $\bar{\lambda}=\bar{\lambda}(\mu )$ and $\lambda ^{\ast }=\lambda
^{\ast }(\mu )$ are strictly increasing and continuous functions with
respect to $\mu >0$. Furthermore,%
\begin{equation*}
\lim_{\mu \rightarrow 0^{+}}\bar{\lambda}(\mu )=4\kappa \text{, \ }\lim_{\mu
\rightarrow \infty }\bar{\lambda}(\mu )=\infty \text{, \ }\lim_{\mu
\rightarrow 0^{+}}\lambda ^{\ast }(\mu )=4\kappa \text{ \ and \ }\lim_{\mu
\rightarrow \infty }\lambda ^{\ast }(\mu )=\infty .
\end{equation*}

\item[(ii)] (\ref{eq1}) has no positive solutions for $\left( \mu ,\lambda
\right) \in M_{0}$, exactly one positive solution for $\left( \mu ,\lambda
\right) \in M_{1}\cup B_{2}$, and exactly two positive solutions for $\left(
\mu ,\lambda \right) \in M_{2}\cup B_{1}$, where%
\begin{equation*}
M_{0}\equiv \{\left( \mu ,\lambda \right) :\mu >0\text{ and }0<\lambda
<\lambda ^{\ast }(\mu )\},
\end{equation*}%
\begin{equation*}
M_{1}\equiv \{\left( \mu ,\lambda \right) :\mu >0\text{ and }\lambda >\bar{%
\lambda}(\mu )\},
\end{equation*}%
\begin{equation*}
M_{2}\equiv \{\left( \mu ,\lambda \right) :\mu >0\text{ and }\lambda ^{\ast
}(\mu )<\lambda <\bar{\lambda}(\mu )\}.
\end{equation*}
\end{itemize}
\end{theorem}

\begin{figure}[h]
\centering
\includegraphics[width=6.1in,height=2.6in,keepaspectratio]{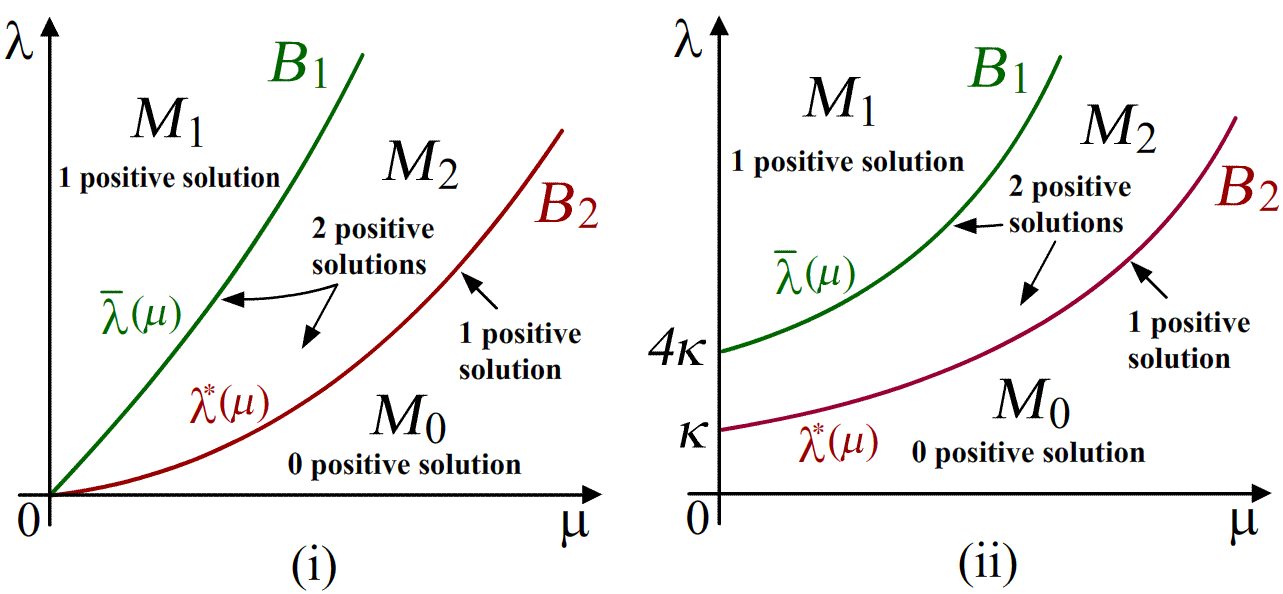}
\caption{The projection of the curves $C_{\Gamma }=C_{1}\cup C_{2}$ onto the
first quadrant of the $\left( \protect\mu ,\protect\lambda \right) $-plane.
(i) $g^{\prime }(0^{+})=\infty $. (ii) $g^{\prime }(0^{+})\in (0,\infty )$.}
\label{fig3}
\end{figure}

The bifurcation set in the $\left( \mu ,\lambda \right) $-plane serves as a
feasibility chart for constant-yield harvesting, distinguishing sustainable
and unsustainable harvesting quotas by the number of positive equilibria. In
particular, the two-solution region signals a fold-type transition and
critical thresholds for harvesting. This work extends the
bifurcation-diagram approach from the semilinear (Laplacian) case to the
Minkowski-curvature diffusion setting, where the underlying time-map
structure is fundamentally altered.

\begin{example}
\label{Ex1}Consider (\ref{eq2}). For the sake of convenience, we let
\begin{equation*}
g(u)=u^{p}\left[ 1-\left( \frac{u}{K}\right) ^{q}\right] ^{r}.
\end{equation*}%
Clearly, $\sigma =K$, $g(0)=g(K)=0$ and $\left( K-u\right) g(u)>0$ on $%
\left( 0,K\right) $. We compute%
\begin{equation*}
g^{\prime }(u)=u^{p-1}\left[ 1-\left( \frac{u}{K}\right) ^{q}\right] ^{r-1}%
\left[ p-\left( p+qr\right) \left( \frac{u}{K}\right) ^{q}\right] .
\end{equation*}%
It follows that $g^{\prime }(u_{0})=0,$\ $g^{\prime }(u)>0$ on $\left(
0,u_{0}\right) $, and $g^{\prime }(u)<0$ on $(u_{0},\sigma )$ where%
\begin{equation*}
u_{0}=K\left( \frac{p}{p+qr}\right) ^{\frac{1}{q}}.
\end{equation*}%
Furthermore, by (\ref{K}),%
\begin{equation*}
\lim_{u\rightarrow 0^{+}}g^{\prime }(u)=\left\{
\begin{array}{ll}
p & \text{if }p=1, \\
\infty  & \text{if }0<p<1.%
\end{array}%
\right. \text{ \ and \ }\kappa =\left\{
\begin{array}{ll}
\frac{\pi ^{2}}{4pL^{2}} & \text{if }p=1, \\
0 & \text{if }0<p<1.%
\end{array}%
\right.
\end{equation*}%
We compute%
\begin{equation}
g^{\prime \prime }(u)=u^{p-2}\left[ 1-\left( \frac{u}{K}\right) ^{q}\right]
^{r-2}\frac{\varphi (u^{q})}{K^{2q}},  \label{d2g}
\end{equation}%
where $\varphi (t)\equiv \left( p+qr\right) \left( p+qr-1\right) t^{2}-K^{q}%
\left[ 2pqr+2p\left( p-1\right) +qr\left( q-1\right) \right] t+K^{2q}p\left(
p-1\right) $. We compute%
\begin{equation}
\varphi (0)=K^{2q}p\left( p-1\right) \text{ \ and \ }\varphi
(K^{q})=K^{2q}q^{2}r\left( r-1\right) .  \label{413}
\end{equation}%
Next, we consider four cases.

\smallskip \textbf{Case 1.} Assume that (h$_{1}$) holds. By (\ref{413}),
then $\varphi (t)$ is a straight line with $\varphi (0)=-K^{2q}pqr<0$ and $%
\varphi (K^{q})\leq 0$. It follows that $\varphi (t)<0$ on $\left(
0,K^{q}\right) .$

\smallskip \textbf{Case 2}. Assume that (h$_{2}$) holds. By (\ref{413}),
then $\varphi (t)$ is a quadratic polynomial with positive leading
coefficient,%
\begin{equation*}
\varphi (0)\leq 0\text{ \ and \ }\varphi (K^{q})\leq 0.
\end{equation*}%
It follows that $\varphi (t)<0$ on $\left( 0,K^{q}\right) .$

\smallskip \textbf{Case 3}. Assume that (h$_{3}$) holds. By (\ref{413}),
then $\varphi (t)$ is a quadratic polynomial with negaive leading
coefficient,%
\begin{equation*}
\varphi (0)<-K^{2q}pqr<0\text{, \ }\varphi (K^{q})\leq 0\text{ \ and \ }%
\varphi ^{\prime }(K^{q})=K^{q}qr\left( 2p-q+2qr-1\right) \geq 0.
\end{equation*}%
It follows that $\varphi (t)<0$ on $\left( 0,K^{q}\right) .$

\smallskip \textbf{Case 4}. Assume that (h$_{4}$) holds. By (\ref{413}),
then $\varphi (t)$ is a quadratic polynomial with negaive leading
coefficient,%
\begin{equation*}
\varphi (0)<-K^{2q}pqr<0\text{, \ }\varphi (K^{q})\leq 0\text{ \ and \ }%
\varphi ^{\prime }(K^{q})=K^{q}qr\left( 2p-q+2qr-1\right) <0.
\end{equation*}%
Let $\Delta $ be the discriminant of $\varphi (t).$ Then%
\begin{equation*}
\Delta =K^{2q}q^{2}r\left[ \left( q-1\right) ^{2}+4pq\right] \left( r-\frac{%
4p\left( 1-p\right) }{\left( q-1\right) ^{2}+4pq}\right) <0.
\end{equation*}%
It follows that $\varphi (t)<0$ on $\left( 0,K^{q}\right) .$

\noindent By Cases 1--4 and (\ref{d2g}), then $g^{\prime \prime }(u)<0$ on $%
\left( 0,K^{q}\right) $. Thus, all results in Theorems \ref{T1}--\ref{T2}
hold. In particular, if $\gamma =1$, then%
\begin{equation*}
\frac{\partial }{\partial u}\frac{G(u)}{u}=\frac{\left( p+q\right) u^{p-1}}{%
K^{q}\left( p+q+1\right) }\left[ \frac{p\left( p+q+1\right) K^{q}}{\left(
p+1\right) \left( p+q\right) }-u^{q}\right] ,
\end{equation*}%
which implies that%
\begin{equation*}
c_{L}^{\ast }=\min \left\{ \left[ \frac{p\left( p+q+1\right) }{\left(
p+1\right) \left( p+q\right) }\right] ^{\frac{1}{q}}K,\text{ }L\right\} .
\end{equation*}
\end{example}

\smallskip

\section{Lemmas}

By Lemma \ref{LH}, we let%
\begin{equation*}
\Omega \equiv \left\{ (\mu ,\lambda ):\lambda >\frac{\mu }{g(c^{\ast })}%
>0\right\} \text{ \ and \ }\mu _{\lambda }\equiv g(c^{\ast })\lambda .
\end{equation*}%
The time-map formula for (\ref{eq1}) is given by%
\begin{equation}
T_{\mu ,\lambda }(\alpha )\equiv \int_{0}^{\alpha }\frac{B(\alpha ,u)+1}{%
\sqrt{B^{2}(\alpha ,u)+2B(\alpha ,u)}}du=\int_{0}^{1}\frac{\alpha \left[
B(\alpha ,\alpha t)+1\right] }{\sqrt{B^{2}(\alpha ,\alpha t)+2B(\alpha
,\alpha t)}}dt  \label{T}
\end{equation}%
for $\theta _{\mu ,\lambda }\leq \alpha <\beta _{\mu ,\lambda }$ and $\left(
\mu ,\lambda \right) \in \Omega $ where
\begin{equation*}
B(\alpha ,u)\equiv F_{\mu ,\lambda }(\alpha )-F_{\mu ,\lambda }(u)=\lambda
\left( G(\alpha )-G(u)\right) -\mu \left( \alpha -u\right) ,
\end{equation*}%
cf. \cite[p. 127]{Corsato} and \cite{Huang1, Huang2}. Observe that positive
solutions $u_{\mu ,\lambda }\in C^{2}(-L,L)\cap C[-L,L]$ for (\ref{eq1})
correspond to%
\begin{equation*}
\left\Vert u_{\mu ,\lambda }\right\Vert _{\infty }=\alpha \text{ \ and \ }%
T_{\mu ,\lambda }(\alpha )=L.
\end{equation*}%
So by (\ref{SL}) and (\ref{CL}), we have that%
\begin{equation}
S_{\mu }=\left\{ \left( \lambda ,\alpha \right) :T_{\mu ,\lambda }(\alpha )=L%
\text{ \ for some }\alpha \in \lbrack \theta _{\mu ,\lambda },\beta _{\mu
,\lambda })\text{ and }\lambda >\lambda _{\mu }\right\} \text{ for }\mu >0
\label{S}
\end{equation}%
\ and%
\begin{equation}
\Sigma _{\lambda }\equiv \left\{ \left( \mu ,\alpha \right) :T_{\mu ,\lambda
}(\alpha )=L\text{ \ for some }\alpha \in \lbrack \theta _{\mu ,\lambda
},\beta _{\mu ,\lambda })\text{ and }\mu >\mu _{\lambda }\right\} \text{ for
}\lambda >0.  \label{C}
\end{equation}%
Understanding the fundamental properties of the time-map function $T_{\mu
,\lambda }(\alpha )$ on $[\theta _{\mu ,\lambda },\beta _{\mu ,\lambda })$
is essential for analyzing the shapes of the bifurcation curves $S_{\mu }$
and $\Sigma _{\lambda }$. Since $g\in C^{2}\left( 0,\infty \right) $, it can
be proved that $T_{\mu ,\lambda }(\alpha )$ is twice continuously
differentiable with respect to $\alpha $, $\lambda $ and $\mu $,
individually. The proofs are straightforward but tedious and hence we omit
them.

\begin{lemma}
\label{Le1}Consider (\ref{eq1}). Then the following statements (i)--(ii)
hold:

\begin{itemize}
\item[(i)] For $\mu >0$, then $\theta _{\mu ,\lambda }$ and $\beta _{\mu
,\lambda }$ are continuously differentiable functions with respect to $%
\lambda \in (\lambda _{\mu },\infty )$. Furthermore,%
\begin{equation}
\frac{\partial \theta _{\mu ,\lambda }}{\partial \lambda }<0\text{, \ }\frac{%
\partial \beta _{\mu ,\lambda }}{\partial \lambda }>0\text{ \ for }\lambda
>\lambda _{\mu }  \label{0f}
\end{equation}%
and
\begin{equation*}
\lim_{\lambda \rightarrow \infty }\theta _{\mu ,\lambda }=0<\theta _{\mu
,\lambda _{\mu }}=\beta _{\mu ,\lambda _{\mu }}=c^{\ast }<\lim_{\lambda
\rightarrow \infty }\beta _{\mu ,\lambda }=\sigma .
\end{equation*}

\item[(ii)] For $\lambda >0$, then $\theta _{\mu ,\lambda }$ and $\beta
_{\mu ,\lambda }$ are continuously differentiable functions with respect to $%
\mu \in (0,\mu _{\lambda })$. Furthermore,%
\begin{equation*}
\frac{\partial \theta _{\mu ,\lambda }}{\partial \mu }>0\text{, \ }\frac{%
\partial \beta _{\mu ,\lambda }}{\partial \mu }<0\text{ \ for }0<\mu <\mu
_{\lambda }
\end{equation*}%
and
\begin{equation}
\lim_{\mu \rightarrow 0^{+}}\theta _{\mu ,\lambda }=0<\theta _{\mu _{\lambda
},\lambda }=\beta _{\mu _{\lambda },\lambda }=c^{\ast }<\lim_{\mu
\rightarrow 0^{+}}\beta _{\mu ,\lambda }=\sigma .  \label{0h}
\end{equation}
\end{itemize}
\end{lemma}

\begin{proof}
(I) Let $\mu >0$ be given. By Lemma \ref{LH}, we observe that, for $\lambda
>\lambda _{\mu }$,%
\begin{equation}
\left\{
\begin{array}{l}
g(\beta _{\mu ,\lambda })>0\text{ \ and \ }g^{\prime }(\beta _{\mu ,\lambda
})<0,\medskip \\
G(\theta _{\mu ,\lambda })>0\text{,}\medskip \\
f_{\mu ,\lambda }(\theta _{\mu ,\lambda })>0\text{, \ }f_{\mu ,\lambda
}(\beta _{\mu ,\lambda })=0\text{ \ and \ }f_{\mu ,\lambda }^{\prime }(\beta
_{\mu ,\lambda })=\lambda g^{\prime }(\beta _{\mu ,\lambda })<0,\medskip \\
F_{\mu ,\lambda }(\theta _{\mu ,\lambda })=0\text{ \ and \ }F_{\mu ,\lambda
}^{\prime }(\theta _{\mu ,\lambda })=f_{\mu ,\lambda }(\theta _{\mu ,\lambda
})>0.%
\end{array}%
\right.  \label{0c}
\end{equation}%
So by implicit function theorem, both $\beta _{\mu ,\lambda }$ and $\theta
_{\mu ,\lambda }$ are continuously differentiable functions with respect to $%
\lambda \in (\lambda _{\mu },\infty )$. Since
\begin{equation}
g(\beta _{\mu ,\lambda })=\frac{\mu }{\lambda }\text{ \ and \ }\lambda
G(\theta _{\mu ,\lambda })=\mu \theta _{\mu ,\lambda }  \label{0g}
\end{equation}%
and by (\ref{0c}), we obtain%
\begin{equation*}
\frac{\partial \beta _{\mu ,\lambda }}{\partial \lambda }=-\frac{\mu }{%
\lambda ^{2}g^{\prime }(\beta _{\mu ,\lambda })}>0\text{ \ and \ }\frac{%
\partial \theta _{\mu ,\lambda }}{\partial \lambda }=-\frac{G(\theta _{\mu
,\lambda })}{f(\theta _{\mu ,\lambda })}<0\text{ \ for }\lambda >\lambda
_{\mu }\text{,}
\end{equation*}%
which implies that (\ref{0f}) holds, and%
\begin{equation}
0<\lim_{\lambda \rightarrow \infty }\beta _{\mu ,\lambda }\leq \sigma \text{
\ and \ }0\leq \lim_{\lambda \rightarrow \infty }\theta _{\mu ,\lambda
}<\sigma .  \label{0b}
\end{equation}%
By (\ref{0g}), we compute%
\begin{equation}
\lim_{\lambda \rightarrow \infty }g(\beta _{\mu ,\lambda })=\lim_{\lambda
\rightarrow \infty }\frac{\mu }{\lambda }=0\text{ \ and \ }\lim_{\lambda
\rightarrow \infty }\frac{G(\theta _{\mu ,\lambda })}{\theta _{\mu ,\lambda }%
}=\lim_{\lambda \rightarrow \infty }\frac{\mu }{\lambda }=0.  \label{0a}
\end{equation}%
By (\ref{0b}) and (\ref{0a}), we conclude that%
\begin{equation*}
\lim_{\lambda \rightarrow \infty }\beta _{\mu ,\lambda }=\sigma \text{ \ and
\ }\lim_{\lambda \rightarrow \infty }\theta _{\mu ,\lambda }=0.
\end{equation*}%
Finally, by (\ref{0g}) and Lemma \ref{LH}(iii), we see that%
\begin{equation*}
\frac{G(\theta _{\mu ,\lambda _{\mu }})}{\theta _{\mu ,\lambda _{\mu }}}=%
\frac{\mu }{\lambda _{\mu }}=g(\beta _{\mu ,\lambda _{\mu }})=g(\theta _{\mu
,\lambda _{\mu }}),
\end{equation*}%
which, by Lemma \ref{LH}(ii), implies that $c^{\ast }=\beta _{\mu ,\lambda
_{\mu }}=\theta _{\mu ,\lambda _{\mu }}$. Thus, the statement (i) holds.

\smallskip

(II) Let $\lambda >0$ be given. By (\ref{0c}), (\ref{0g}) and implicit
function theorem, both $\theta _{\mu ,\lambda }$ and $\beta _{\mu ,\lambda }$
are continuously differentiable functions with respect to $\mu \in (0,\mu
_{0})$,%
\begin{equation}
\frac{\partial \beta _{\mu ,\lambda }}{\partial \mu }=\frac{1}{\lambda
g^{\prime }(\beta _{\mu ,\lambda })}<0\text{ \ and \ }\frac{\partial \theta
_{\mu ,\lambda }}{\partial \mu }=\frac{\theta _{\mu ,\lambda }}{f(\theta
_{\mu ,\lambda })}>0\text{ \ for }\mu \in (0,\mu _{0}).  \label{0j}
\end{equation}%
Since the proof of (\ref{0h}) follows a similar argument in (I), we omit the
detail. Thus, the statement (ii) holds. The proof is complete.
\end{proof}

\begin{lemma}[See Figure \protect\ref{fig8}]
\label{Le0}Consider (\ref{eq2}). For any $\lambda >0$, the following
statements (i)--(ii) hold.

\begin{itemize}
\item[(i)] $T_{0,\lambda }^{\prime }(\alpha )>0$ on $(0,\sigma )$, $%
T_{0,\lambda }(0^{+})=\eta $ and $T_{0,\lambda }(\sigma ^{-})=\infty $ where
$\eta $ is defined by (\ref{E}). Moreover, if $L>\eta $, there exists unique
$\gamma _{\lambda }\in \left( 0,\sigma \right) $ such that $T_{0,\lambda
}(\gamma _{\lambda })=L$.

\item[(ii)] $T_{\mu ,\lambda }(\alpha )>\lim_{\mu \rightarrow 0^{+}}T_{\mu
,\lambda }(\alpha )=T_{0,\lambda }(\alpha )$ for $\theta _{\mu ,\lambda
}<\alpha <\beta _{\mu ,\lambda }$ and $\mu \in (0,\mu _{\lambda }).$
\end{itemize}
\end{lemma}

\begin{proof}
Let $\lambda >0$ be given. By \cite[Lemmas 4.1 and 4.2]{Huang2}, we obtain $%
T_{0,\lambda }(0^{+})=\eta $ and $T_{0,\lambda }(\sigma ^{-})=\infty $.
Since $g(0)=0$ and $g^{\prime \prime }(u)<0$ on $\left( 0,\sigma \right) $,
and by Mean-value theorem, we see that, for any $u\in (0,\sigma )$, there
exists $z_{u}\in \left( 0,u\right) $ such that%
\begin{equation}
g(u)=ug^{\prime }(z_{u})>ug^{\prime }(u)>0,  \label{gg}
\end{equation}%
from which it follows that
\begin{equation*}
\left( \frac{g(u)}{u}\right) ^{\prime }=\frac{ug^{\prime }(u)-g(u)}{u^{2}}<0%
\text{ \ for }u\in (0,\sigma ).
\end{equation*}%
So by \cite[p.3456]{Huang2}, we conclude that $T_{0,\lambda }^{\prime
}(\alpha )>0$ on $(0,\sigma )$. Thus, if $L>\eta $, there exists unique $%
\gamma _{\lambda }\in \left( 0,\sigma \right) $ such that $T_{0,\lambda
}(\gamma _{\lambda })=L$. The statement (i) holds.

We compute
\begin{equation}
\frac{\partial }{\partial \mu }T_{\mu ,\lambda }(\alpha )=\int_{0}^{\alpha }%
\frac{\alpha -u}{\left[ B(\alpha ,u)+2B(\alpha ,u)\right] ^{3/2}}du>0
\label{Le4b}
\end{equation}%
\ for $\theta _{\mu ,\lambda }\leq \alpha <\beta _{\mu ,\lambda }$ and $\mu
\in (0,\mu _{\lambda }).$ By (\ref{T}), (\ref{Le4b}) and Monotone
convergence theorem, we obtain%
\begin{eqnarray*}
T_{\mu ,\lambda }(\alpha ) &>&\lim_{\mu \rightarrow 0^{+}}T_{\mu ,\lambda
}(\alpha ) \\
&=&\int_{0}^{\alpha }\frac{\lambda \left[ G(\alpha )-G(u)\right] +1}{\sqrt{%
\lambda ^{2}\left[ G(\alpha )-G(u)\right] ^{2}+2\lambda \left[ G(\alpha
)-G(u)\right] }}du \\
&=&T_{0,\lambda }(\alpha )
\end{eqnarray*}%
for $\theta _{\mu ,\lambda }\leq \alpha <\beta _{\mu ,\lambda }$ and $\mu
\in (0,\mu _{\lambda })$. The statement (ii) holds. The proof is complete.
\end{proof}

\begin{figure}[h]
\centering
\includegraphics[width=6.1in,height=2.6in,keepaspectratio]{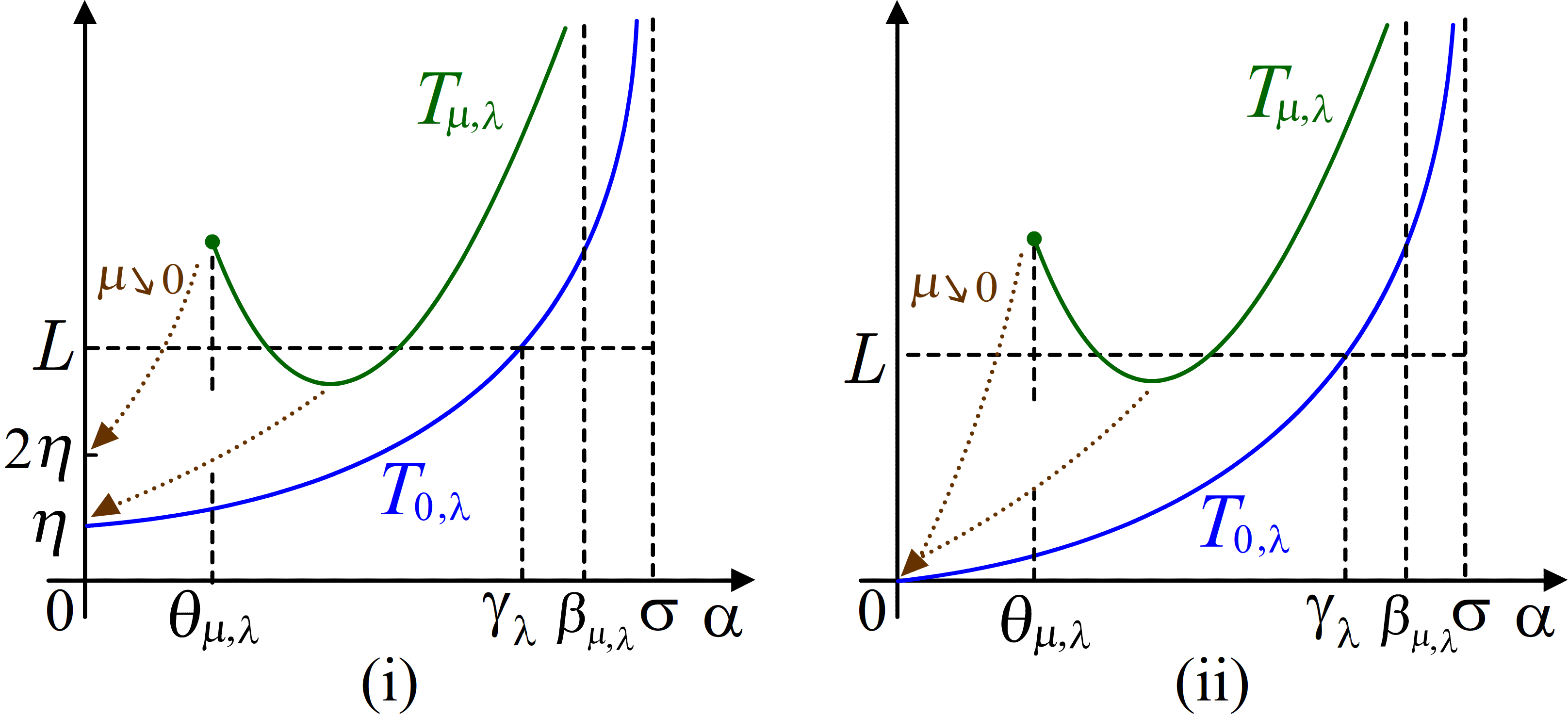}
\caption{Graphs of $T_{0,\protect\lambda }$ and $T_{\protect\mu ,\protect%
\lambda }$. (i) $\protect\eta >0$ (i.e. $g^{\prime }(0^{+})\in (0,\infty )$%
). (ii) $\protect\eta =0$ (i.e. $g^{\prime }(0^{+})=\infty $).}
\label{fig8}
\end{figure}

\begin{lemma}
\label{Le2}Consider (\ref{eq1}). Then $T_{\mu ,\lambda }(\theta _{\mu
,\lambda }^{+})\in (0,\infty ),$ $T_{\mu ,\lambda }^{\prime }(\theta _{\mu
,\lambda }^{+})=-\infty $ and $T_{\mu ,\lambda }(\beta _{\mu ,\lambda
}^{-})=\infty $ for $(\mu ,\lambda )\in \Omega .$
\end{lemma}

\begin{proof}
It is easy to compute that%
\begin{equation*}
\lim_{u\rightarrow 0^{+}}\frac{f_{\mu ,\lambda }(u)}{\sqrt{u}}=-\infty \text{
\ and \ }\lim_{u\rightarrow 0^{+}}u^{\frac{1}{3}}f_{\mu ,\lambda }(u)=0.
\end{equation*}%
The proof is complete by \cite[Lemmas 4.1 and 4.2]{Huang1} and \cite[Lemma
4.2]{Huang2}.
\end{proof}

\begin{lemma}
\label{Le3}Consider (\ref{eq1}). For $(\mu ,\lambda )\in \Omega $, there
exists $\tilde{\alpha}_{\mu ,\lambda }\in (\theta _{\mu ,\lambda },\beta
_{\mu ,\lambda })$ such that%
\begin{equation}
T_{\mu ,\lambda }^{\prime }(\alpha )\left\{
\begin{array}{ll}
<0 & \text{for }\theta _{\mu ,\lambda }<\alpha <\tilde{\alpha}_{\mu ,\lambda
},\smallskip \\
=0 & \text{for }\alpha =\tilde{\alpha}_{\mu ,\lambda },\smallskip \\
>0 & \text{for }\tilde{\alpha}_{\mu ,\lambda }<\alpha <\beta _{\mu ,\lambda
},%
\end{array}%
\right. \text{ \ and \ }T_{\mu ,\lambda }^{\prime \prime }(\tilde{\alpha}%
_{\mu ,\lambda })>0.  \label{CS}
\end{equation}
\end{lemma}

\begin{proof}
Let $(\mu ,\lambda )\in \Omega $ be given. By Lemma \ref{Le2}, $T_{\mu
,\lambda }(\alpha )$ has at least one critical number, a local minimum, on $%
(\theta _{\mu ,\lambda },\beta _{\mu ,\lambda })$. Since $f_{\mu ,\lambda
}^{\prime \prime }(u)=\lambda g^{\prime \prime }(u)<0$ for $0<u<\sigma $,
and by \cite[Lemma 4.7]{Huang1}, we obtain%
\begin{equation*}
T_{\mu ,\lambda }^{\prime \prime }(\alpha )+\frac{2}{\alpha }T_{\mu ,\lambda
}^{\prime }(\alpha )>0\text{ \ for }\theta _{\mu ,\lambda }<\alpha <\beta
_{\mu ,\lambda }\text{,}
\end{equation*}%
from which it follows that $T_{\mu ,\lambda }^{\prime \prime }(\alpha )>0$
for any critical number $\alpha \in (\theta _{\mu ,\lambda },\beta _{\mu
,\lambda })$. Consequently, $T_{\mu ,\lambda }(\alpha )$ has exactly one
critical number $\tilde{\alpha}_{\mu ,\lambda }$, a local minimum, on $%
(\theta _{\mu ,\lambda },\beta _{\mu ,\lambda })$, and $T_{\mu ,\lambda
}^{\prime \prime }(\tilde{\alpha}_{\mu ,\lambda })>0$. The proof is complete.
\end{proof}

\begin{lemma}
\label{Le4}Consider (\ref{eq1}). Let $\eta $ be defined in (\ref{E}). Then
the following statements (i)--(ii) hold.

\begin{itemize}
\item[(i)] For any $\mu >0,$

\begin{itemize}
\item[(a)] $T_{\mu ,\lambda }(\theta _{\mu ,\lambda })$ is a continuous
function with respect to $\lambda \in (\lambda _{\mu },\infty )$. Moreover,%
\begin{equation*}
\lim_{\lambda \rightarrow \lambda _{\mu }^{+}}T_{\mu ,\lambda }(\theta _{\mu
,\lambda })=\infty \text{ \ and \ }\lim_{\lambda \rightarrow \infty }T_{\mu
,\lambda }(\theta _{\mu ,\lambda })=0.
\end{equation*}

\item[(b)] $T_{\mu ,\lambda }(\tilde{\alpha}_{\mu ,\lambda })$ is a strictly
decreasing and continuous function with respect to $\lambda \in (\lambda
_{\mu },\infty )$. Moreover,%
\begin{equation*}
\lim_{\lambda \rightarrow \lambda _{\mu }^{+}}T_{\mu ,\lambda }(\tilde{\alpha%
}_{\mu ,\lambda })=\infty \text{ \ and \ }\lim_{\lambda \rightarrow \infty
}T_{\mu ,\lambda }(\tilde{\alpha}_{\mu ,\lambda })=0.
\end{equation*}
\end{itemize}

\item[(ii)] For any $\lambda >0,$

\begin{itemize}
\item[(a)] $T_{\mu ,\lambda }(\theta _{\mu ,\lambda })$ is a continuous
function with respect to $\mu \in (0,\mu _{\lambda })$. Moreover,%
\begin{equation*}
\lim_{\mu \rightarrow 0^{+}}T_{\mu ,\lambda }(\theta _{\mu ,\lambda })=2\eta
\text{ \ and \ }\lim_{\mu \rightarrow \mu _{\lambda }^{-}}T_{\mu ,\lambda
}(\theta _{\mu ,\lambda })=\infty .
\end{equation*}

\item[(b)] $T_{\mu ,\lambda }(\tilde{\alpha}_{\mu ,\lambda })$ is a strictly
increasing and continuous function with respect to $\mu \in (0,\mu _{\lambda
})$. Moreover,
\begin{equation*}
\lim_{\mu \rightarrow 0^{+}}T_{\mu ,\lambda }(\tilde{\alpha}_{\mu ,\lambda
})=\eta \text{ \ and \ }\lim_{\mu \rightarrow \mu _{\lambda }^{-}}T_{\mu
,\lambda }(\tilde{\alpha}_{\mu ,\lambda })=\infty .
\end{equation*}
\end{itemize}

\noindent See Figure \ref{fig8}.
\end{itemize}
\end{lemma}

\begin{proof}
(I) Let $\mu >0$ be given. We divide the proof of the statement (i) into the
following three steps.

\smallskip \textbf{Step 1}. We prove that $T_{\mu ,\lambda }(\theta _{\mu
,\lambda })$ is continuous, and $T_{\mu ,\lambda }(\tilde{\alpha}_{\mu
,\lambda })$ is strictly decreasing and continuous with respect to $\lambda
\in (\lambda _{\mu },\infty )$. By Lemma \ref{Le1}, it is easy to see that $%
T_{\mu ,\lambda }(\theta _{\mu ,\lambda })$ is continuous with respect to $%
\lambda \in (\lambda _{\mu },\infty )$. By (\ref{CS}) and implicit function
theorem, $\tilde{\alpha}_{\mu ,\lambda }$ is continuously differentiable
with respect to $\lambda \in (\lambda _{\mu },\infty )$. Consequently, $%
T_{\mu ,\lambda }(\tilde{\alpha}_{\mu ,\lambda })$ is also continuously
differentiable with respect to $\lambda \in (\lambda _{\mu },\infty )$.
Since $G^{\prime }(u)=g(u)>0$ on $\left( 0,\sigma \right) $, we see that%
\begin{equation}
\frac{\partial }{\partial \lambda }T_{\mu ,\lambda }(\alpha
)=\int_{0}^{\alpha }\frac{-\left[ G(\alpha )-G(u)\right] }{\left[ B(\alpha
,u)+2B(\alpha ,u)\right] ^{3/2}}du<0\text{,}  \label{Le4a}
\end{equation}%
\ for $\theta _{\mu ,\lambda }\leq \alpha <\beta _{\mu ,\lambda }$ and $%
\lambda >\lambda _{\mu }$, from which it follows that%
\begin{eqnarray}
\frac{\partial }{\partial \lambda }T_{\mu ,\lambda }(\tilde{\alpha}_{\mu
,\lambda }) &=&T_{\mu ,\lambda }^{\prime }(\tilde{\alpha}_{\mu ,\lambda })%
\frac{\partial \tilde{\alpha}_{\mu ,\lambda }}{\partial \lambda }+\left[
\frac{\partial }{\partial \lambda }T_{\mu ,\lambda }(\alpha )\right]
_{\alpha =\tilde{\alpha}_{\mu ,\lambda }}  \notag \\
&=&\left[ \frac{\partial }{\partial \lambda }T_{\mu ,\lambda }(\alpha )%
\right] _{\alpha =\tilde{\alpha}_{\mu ,\lambda }}<0\text{ \ for }\lambda
>\lambda _{\mu }.  \label{Le4e}
\end{eqnarray}%
Then $T_{\mu ,\lambda }(\tilde{\alpha}_{\mu ,\lambda })$ is strictly
decreasing with respect to $\lambda \in (\lambda _{\mu },\infty )$.

\smallskip \textbf{Step 2.} We prove that $\lim\limits_{\lambda \rightarrow
\infty }T_{\mu ,\lambda }(\theta _{\mu ,\lambda })=\lim\limits_{\lambda
\rightarrow \infty }T_{\mu ,\lambda }(\tilde{\alpha}_{\mu ,\lambda })=0$.
Since $F_{\mu ,\lambda }(\theta _{\mu ,\lambda })=0$ for $\lambda >\lambda
_{\mu }$, we have%
\begin{equation}
\lambda G(\theta _{\mu ,\lambda })=\mu \theta _{\mu ,\lambda }\text{ \ for }%
\lambda >\lambda _{\mu },  \label{Le4j}
\end{equation}%
from which it follows that%
\begin{equation}
B(\theta _{\mu ,\lambda },\theta _{\mu ,\lambda }t)=-F_{\mu ,\lambda
}(\theta _{\mu ,\lambda }t)=\mu \theta _{\mu ,\lambda }t-\lambda G(\theta
_{\mu ,\lambda }t)=\lambda E_{t}(\theta _{\mu ,\lambda }),  \label{t2}
\end{equation}%
for $0<t<1$, where%
\begin{equation}
E_{t}(\alpha )\equiv G(\alpha )t-G(\alpha t).  \label{ET}
\end{equation}%
By Lemma \ref{LH}(i), we observe that%
\begin{equation*}
B(\alpha ,0)=B(\alpha ,\alpha )=0
\end{equation*}%
and%
\begin{equation*}
\frac{\partial }{\partial u}B(\alpha ,u)=-f_{\mu ,\lambda }(u)\left\{
\begin{array}{ll}
>0 & \text{for }0<u<\varsigma _{\mu ,\lambda }, \\
=0 & \text{for }u=\varsigma _{\mu ,\lambda }, \\
<0 & \text{for }\varsigma _{\mu ,\lambda }<u<\beta _{\mu ,\lambda }.%
\end{array}%
\right.
\end{equation*}%
Then we obtain that
\begin{equation}
B(\alpha ,u)>0\text{ \ for }0<u<\alpha \text{ and }\theta _{\mu ,\lambda
}\leq \alpha <\beta _{\mu ,\lambda }.  \label{B}
\end{equation}%
Since $G(u)>0$ on $\left( 0,\sigma \right) $, and by (\ref{t2}), (\ref{B})
and Lemma \ref{Le1}, we obtain%
\begin{equation*}
0\leq \lim_{\lambda \rightarrow \infty }B(\theta _{\mu ,\lambda },\theta
_{\mu ,\lambda }t)\leq \lim_{\lambda \rightarrow \infty }\mu \theta _{\mu
,\lambda }t=0,
\end{equation*}%
which implies that%
\begin{equation}
\lim_{\lambda \rightarrow \infty }B(\theta _{\mu ,\lambda },\theta _{\mu
,\lambda }t)=0\text{ \ for }0<t<1.  \label{Le4g}
\end{equation}%
By Lemma \ref{Le1}, L'H\^{o}pital's rule and Mean-value theorem, we observe
that, for $0<t<1$,%
\begin{eqnarray*}
\lim_{\lambda \rightarrow \infty }\frac{E_{t}(\theta _{\mu ,\lambda })}{%
\theta _{\mu ,\lambda }^{2}} &=&\lim_{v\rightarrow 0^{+}}\frac{G(v)t-G(vt)}{%
v^{2}}=\lim_{v\rightarrow 0^{+}}\frac{t\left( 1-t\right) }{2}\frac{g(v)-g(vt)%
}{v\left( 1-t\right) } \\
&=&\lim_{v\rightarrow 0^{+}}\frac{t\left( 1-t\right) }{2}g^{\prime }(v_{t})%
\text{ \ for some }v_{t}\in (vt,v) \\
&>&0\text{ \ (since }g^{\prime \prime }(u)<0\text{ on }\left( 0,\sigma
\right) \text{).}
\end{eqnarray*}%
So by (\ref{t2}),%
\begin{equation}
\lim_{\lambda \rightarrow \infty }\frac{B(\theta _{\mu ,\lambda },\theta
_{\mu ,\lambda }t)}{\theta _{\mu ,\lambda }^{2}}=\lim_{\lambda \rightarrow
\infty }\frac{\lambda E_{t}(\theta _{\mu ,\lambda })}{\theta _{\mu ,\lambda
}^{2}}=\infty \text{ \ for }0<t<1.  \label{Le4h}
\end{equation}%
By (\ref{Le4g}) and (\ref{Le4h}), we obtain%
\begin{equation*}
\lim_{\lambda \rightarrow \infty }T_{\mu ,\lambda }(\theta _{\mu ,\lambda
})=\lim_{\lambda \rightarrow \infty }\int_{0}^{1}\frac{B(\theta _{\mu
,\lambda },\theta _{\mu ,\lambda }t)+1}{\sqrt{B(\theta _{\mu ,\lambda
},\theta _{\mu ,\lambda }t)+2}}\frac{\theta _{\mu ,\lambda }}{\sqrt{B(\theta
_{\mu ,\lambda },\theta _{\mu ,\lambda }t)}}dt=0.
\end{equation*}%
So by Lemma \ref{Le3}, we deduce%
\begin{equation*}
0\leq \lim_{\lambda \rightarrow \infty }T_{\mu ,\lambda }(\tilde{\alpha}%
_{\mu ,\lambda })\leq \lim_{\lambda \rightarrow \infty }T_{\mu ,\lambda
}(\theta _{\mu ,\lambda })=0,
\end{equation*}%
which implies that $\lim\limits_{\lambda \rightarrow \infty }T_{\mu ,\lambda
}(\theta _{\mu ,\lambda })=\lim\limits_{\lambda \rightarrow \infty }T_{\mu
,\lambda }(\tilde{\alpha}_{\mu ,\lambda })=0$.

\smallskip \textbf{Step 3.} We prove that $\lim\limits_{\lambda \rightarrow
\lambda _{\mu }^{+}}T_{\mu ,\lambda }(\theta _{\mu ,\lambda
})=\lim\limits_{\lambda \rightarrow \lambda _{\mu }^{+}}T_{\mu ,\lambda }(%
\tilde{\alpha}_{\mu ,\lambda })=\infty $. By Lemma \ref{LH}, we have%
\begin{equation}
\frac{G(c^{\ast })}{c^{\ast }}=g(c^{\ast })\text{ \ and \ }\lambda _{\mu }=%
\frac{\mu }{g(c^{\ast })}=\frac{\mu c^{\ast }}{G(c^{\ast })}.  \label{4a}
\end{equation}%
Since $\theta _{\mu ,\lambda }<\tilde{\alpha}_{\mu ,\lambda }<\beta _{\mu
,\lambda }$ for $\lambda >\lambda _{\mu }$, and by Lemma \ref{Le1}, we
obtain $\lim_{\lambda \rightarrow \lambda _{\mu }^{+}}\tilde{\alpha}_{\mu
,\lambda }=c^{\ast }$, which, by (\ref{4a}), implies that, for $0<t<1,$%
\begin{equation*}
\lim_{\lambda \rightarrow \lambda _{\mu }^{+}}F(\tilde{\alpha}_{\mu ,\lambda
})=\lim_{\lambda \rightarrow \lambda _{\mu }^{+}}\left( \lambda G(\tilde{%
\alpha}_{\mu ,\lambda })-\mu \tilde{\alpha}_{\mu ,\lambda }\right) =\lambda
_{\mu }G(c^{\ast })-\mu c^{\ast }=0
\end{equation*}%
and%
\begin{equation*}
\lim_{\lambda \rightarrow \lambda _{\mu }^{+}}F(\tilde{\alpha}_{\mu ,\lambda
}t)=\lim_{\lambda \rightarrow \lambda _{\mu }^{+}}\left( \lambda G(\tilde{%
\alpha}_{\mu ,\lambda }t)-\mu \tilde{\alpha}_{\mu ,\lambda }t\right) =-\frac{%
c^{\ast }\mu }{G(c^{\ast })}E_{t}(c^{\ast }).
\end{equation*}%
Therefore,%
\begin{equation}
\lim_{\lambda \rightarrow \lambda _{\mu }^{+}}B(\tilde{\alpha}_{\mu ,\lambda
},\tilde{\alpha}_{\mu ,\lambda }t)=\frac{c^{\ast }\mu }{G(c^{\ast })}%
E_{t}(c^{\ast })\text{ \ for }0<t<1.  \label{4f}
\end{equation}%
By L'H\^{o}pital's rule and Lemma \ref{LH}, we see that%
\begin{eqnarray}
\lim_{t\rightarrow 1^{-}}\frac{E_{t}(c^{\ast })}{\left( 1-t\right) ^{2}}
&=&\lim_{t\rightarrow 1^{-}}\frac{G(c^{\ast })-c^{\ast }g(c^{\ast }t)}{%
-2\left( 1-t\right) }=\lim_{t\rightarrow 1^{-}}\frac{-\left( c^{\ast
}\right) ^{2}g^{\prime }(c^{\ast }t)}{2}  \notag \\
&=&\frac{-\left( c^{\ast }\right) ^{2}g^{\prime }(c^{\ast })}{2}\in
(0,\infty ).  \label{4d}
\end{eqnarray}%
By (\ref{4f}) and (\ref{4d}), there exist $\tilde{M}>0$ and $\delta \in
\left( 0,1\right) $ such that%
\begin{equation*}
\lim_{\lambda \rightarrow \lambda _{\mu }^{+}}B(\tilde{\alpha}_{\mu ,\lambda
},\tilde{\alpha}_{\mu ,\lambda }t)<\tilde{M}\left( 1-t\right) ^{2}<1\text{ \
for }\delta <t<1\text{,}
\end{equation*}%
from which it follows that
\begin{equation}
\lim_{\lambda \rightarrow \lambda _{\mu }^{+}}\left[ B^{2}(\tilde{\alpha}%
_{\mu ,\lambda },\tilde{\alpha}_{\mu ,\lambda }t)+2B(\tilde{\alpha}_{\mu
,\lambda },\tilde{\alpha}_{\mu ,\lambda }t)\right] \leq 3\lim_{\lambda
\rightarrow \lambda _{\mu }^{+}}B(\tilde{\alpha}_{\mu ,\lambda },\tilde{%
\alpha}_{\mu ,\lambda }t)\leq 3\tilde{M}\left( 1-t\right) ^{2}.  \label{4c}
\end{equation}%
By (\ref{T}) and (\ref{4c}), we obtain%
\begin{eqnarray}
\lim_{\lambda \rightarrow \lambda _{\mu }^{+}}T_{\mu ,\lambda }(\tilde{\alpha%
}_{\mu ,\lambda }) &=&\lim_{\lambda \rightarrow \lambda _{\mu
}^{+}}\int_{0}^{1}\frac{\tilde{\alpha}_{\mu ,\lambda }\left[ B(\tilde{\alpha}%
_{\mu ,\lambda },\tilde{\alpha}_{\mu ,\lambda }t)+1\right] }{\sqrt{B^{2}(%
\tilde{\alpha}_{\mu ,\lambda },\tilde{\alpha}_{\mu ,\lambda }t)+2B(\tilde{%
\alpha}_{\mu ,\lambda },\tilde{\alpha}_{\mu ,\lambda }t)}}dt  \notag \\
&\geq &\lim_{\lambda \rightarrow \lambda _{\mu }^{+}}\int_{\delta }^{1}\frac{%
\tilde{\alpha}_{\mu ,\lambda }}{\sqrt{B^{2}(\tilde{\alpha}_{\mu ,\lambda },%
\tilde{\alpha}_{\mu ,\lambda }t)+2B(\tilde{\alpha}_{\mu ,\lambda },\tilde{%
\alpha}_{\mu ,\lambda }t)}}dt  \notag \\
&\geq &\frac{c^{\ast }}{\sqrt{3\tilde{M}}}\int_{\delta }^{1}\frac{1}{1-t}%
dt=\infty .  \label{4e}
\end{eqnarray}%
So by Lemma \ref{Le3}, then%
\begin{equation*}
\lim_{\lambda \rightarrow \lambda _{\mu }^{+}}T_{\mu ,\lambda }(\theta _{\mu
,\lambda })\geq \lim_{\lambda \rightarrow \lambda _{\mu }^{+}}T_{\mu
,\lambda }(\tilde{\alpha}_{\mu ,\lambda })=\infty ,
\end{equation*}%
which implies that $\lim\limits_{\lambda \rightarrow \lambda _{\mu
}^{+}}T_{\mu ,\lambda }(\theta _{\mu ,\lambda })=\lim\limits_{\lambda
\rightarrow \lambda _{\mu }^{+}}T_{\mu ,\lambda }(\tilde{\alpha}_{\mu
,\lambda })=\infty .$

\smallskip

(II) Let $\lambda >0$ be given. We divide the proof of the statement (ii)
into the following three steps.

\smallskip \textbf{Step 1}. We prove that $T_{\mu ,\lambda }(\theta _{\mu
,\lambda })$ is continuous, and $T_{\mu ,\lambda }(\tilde{\alpha}_{\mu
,\lambda })$ is strictly increasing and continuous with respect to $\mu \in
(0,\mu _{\lambda })$. By Lemma \ref{Le1}, it is easy to see that $T_{\mu
,\lambda }(\theta _{\mu ,\lambda })$ is continuous with respect to $\mu \in
(0,\mu _{\lambda })$. By (\ref{CS}) and implicit function theorem, $\tilde{%
\alpha}_{\mu ,\lambda }$ is continuously differentiable with respect to $\mu
\in (0,\mu _{\lambda })$. Consequently, $T_{\mu ,\lambda }(\tilde{\alpha}%
_{\mu ,\lambda })$ is also continuously differentiable with respect to $\mu
\in (0,\mu _{\lambda })$. By (\ref{Le4b}), we observe that%
\begin{eqnarray*}
\frac{\partial }{\partial \mu }T_{\mu ,\lambda }(\tilde{\alpha}_{\mu
,\lambda }) &=&T_{\mu ,\lambda }^{\prime }(\tilde{\alpha}_{\mu ,\lambda })%
\frac{\partial \tilde{\alpha}_{\mu ,\lambda }}{\partial \mu }+\left[ \frac{%
\partial }{\partial \mu }T_{\mu ,\lambda }(\alpha )\right] _{\alpha =\tilde{%
\alpha}_{\mu ,\lambda }} \\
&=&\left[ \frac{\partial }{\partial \mu }T_{\mu ,\lambda }(\alpha )\right]
_{\alpha =\tilde{\alpha}_{\mu ,\lambda }}>0\text{ \ for }\mu \in (0,\mu
_{\lambda }),
\end{eqnarray*}%
which implies that $T_{\mu ,\lambda }(\tilde{\alpha}_{\mu ,\lambda })$ is
strictly increasing with respect to $\mu \in (0,\mu _{\lambda })$.

\smallskip \textbf{Step 2.} We prove $\lim\limits_{\mu \rightarrow
0^{+}}T_{\mu ,\lambda }(\theta _{\mu ,\lambda })=2\eta $. By (\ref{0j}) and
Lemma \ref{Le1}, we have%
\begin{equation}
\lim_{\mu \rightarrow 0^{+}}\theta _{\mu ,\lambda }=0\text{ \ and \ }\frac{%
\partial \theta _{\mu ,\lambda }}{\partial \mu }=\frac{\theta _{\mu ,\lambda
}}{f_{\mu ,\lambda }(\theta _{\mu ,\lambda })}=\frac{\theta _{\mu ,\lambda }%
}{\lambda g(\theta _{\mu ,\lambda })-\mu }.  \label{Le4m}
\end{equation}%
It follows that%
\begin{equation}
\lim_{\mu \rightarrow 0^{+}}B(\theta _{\mu ,\lambda },\theta _{\mu ,\lambda
}t)=-\lim_{\mu \rightarrow 0^{+}}F_{\mu ,\lambda }(\theta _{\mu ,\lambda
}t)=\lim_{\mu \rightarrow 0^{+}}\left[ \mu \theta _{\mu ,\lambda }t-\lambda
G(\theta _{\mu ,\lambda }t)\right] =0.  \label{Le4n}
\end{equation}%
By (\ref{Le4n}), L'H\^{o}pital's rule and Mean-value theorem,%
\begin{eqnarray}
\lim_{\mu \rightarrow 0^{+}}\frac{B(\theta _{\mu ,\lambda },\theta _{\mu
,\lambda }t)}{\theta _{\mu ,\lambda }^{2}} &=&\lim_{\mu \rightarrow 0^{+}}%
\frac{\theta _{\mu ,\lambda }t+t\left[ \mu -\lambda g(\theta _{\mu ,\lambda
}t)\right] \frac{\partial \theta _{\mu ,\lambda }}{\partial \mu }}{2\theta
_{\mu ,\lambda }\frac{\partial \theta _{\mu ,\lambda }}{\partial \mu }}
\notag \\
&=&\lim_{\mu \rightarrow 0^{+}}\frac{g(\theta _{\mu ,\lambda })-g(\theta
_{\mu ,\lambda }t)}{2\theta _{\mu ,\lambda }}\lambda t\text{ \ (by (\ref%
{Le4m}))}  \notag \\
&=&\lim_{v\rightarrow 0^{+}}\frac{\lambda t\left( 1-t\right) }{2}\frac{%
g(v)-g(vt)}{v-vt}  \notag \\
&=&\lim_{v\rightarrow 0^{+}}\frac{\lambda t\left( 1-t\right) }{2}g^{\prime
}(v_{t})\text{ \ for some }v_{t}\in (vt,v)  \notag \\
&=&\left\{
\begin{array}{ll}
\infty & \text{if }g^{\prime }(0^{+})=\infty , \\
\frac{\lambda t\left( 1-t\right) }{2}g^{\prime }(0^{+}) & \text{if }%
g^{\prime }(0^{+})\in (0,\infty ).%
\end{array}%
\right.  \label{Le4c}
\end{eqnarray}%
By (\ref{Le4n}) and (\ref{Le4c}), we observe that%
\begin{eqnarray*}
\lim_{\mu \rightarrow 0^{+}}T_{\mu ,\lambda }(\theta _{\mu ,\lambda })
&=&\lim_{\mu \rightarrow 0^{+}}\int_{0}^{1}\frac{B(\theta _{\mu ,\lambda
},\theta _{\mu ,\lambda }t)+1}{\sqrt{B(\theta _{\mu ,\lambda },\theta _{\mu
,\lambda }t)+2}}\frac{\theta _{\mu ,\lambda }}{\sqrt{B(\theta _{\mu ,\lambda
},\theta _{\mu ,\lambda }t)}}dt\smallskip \\
&=&\left\{
\begin{array}{ll}
0 & \text{if }g^{\prime }(0^{+})=\infty ,\smallskip \\
\frac{1}{\sqrt{\lambda g^{\prime }(0^{+})}}\int_{0}^{1}\frac{1}{\sqrt{%
t\left( 1-t\right) }}dt=\frac{\pi }{\sqrt{\lambda g^{\prime }(0^{+})}} &
\text{if }g^{\prime }(0^{+})\in (0,\infty ).%
\end{array}%
\right.
\end{eqnarray*}%
So $\lim\limits_{\mu \rightarrow 0^{+}}T_{\mu ,\lambda }(\theta _{\mu
,\lambda })=2\eta .$

\smallskip \textbf{Step 3.} We prove the statement (ii). By Steps 1--2, it
is sufficient to prove that%
\begin{equation}
\lim_{\mu \rightarrow 0^{+}}T_{\mu ,\lambda }(\tilde{\alpha}_{\mu ,\lambda
})=\eta \text{, \ and \ }\lim_{\mu \rightarrow \mu _{\lambda }^{-}}T_{\mu
,\lambda }(\theta _{\mu ,\lambda })=\lim_{\mu \rightarrow \mu _{\lambda
}^{-}}T_{\mu ,\lambda }(\tilde{\alpha}_{\mu ,\lambda })=\infty .  \label{4i}
\end{equation}%
Take $\breve{\alpha}\in \left( 0,c^{\ast }\right) $. By Lemma \ref{Le1}(ii),
then $\breve{\alpha}\in \left( \theta _{\mu ,\lambda },\beta _{\mu ,\lambda
}\right) $ for all sufficiently small $\mu >0$, which, by Lemmas \ref{Le0}
and \ref{Le3}, implies that%
\begin{equation}
T_{0,\lambda }(\breve{\alpha})=\lim_{\mu \rightarrow 0^{+}}T_{\mu ,\lambda }(%
\breve{\alpha})\geq \lim_{\mu \rightarrow 0^{+}}T_{\mu ,\lambda }(\tilde{%
\alpha}_{\mu ,\lambda })\geq \lim_{\mu \rightarrow 0^{+}}T_{0,\lambda }(%
\tilde{\alpha}_{\mu ,\lambda })\geq \eta .  \label{4ib}
\end{equation}%
Since $\breve{\alpha}$ is arbitrary, we take $\breve{\alpha}\rightarrow
0^{+} $. So by (\ref{4ib}) and Lemma \ref{Le0}, we obtain%
\begin{equation}
\lim_{\mu \rightarrow 0^{+}}T_{\mu ,\lambda }(\tilde{\alpha}_{\mu ,\lambda
})=\eta .  \label{4j}
\end{equation}%
In addition, by the similar argument in (\ref{4e}), we see that%
\begin{equation}
\lim_{\mu \rightarrow \mu _{\lambda }^{-}}T_{\mu ,\lambda }(\theta _{\mu
,\lambda })\geq \lim_{\mu \rightarrow \mu _{\lambda }^{-}}T_{\mu ,\lambda }(%
\tilde{\alpha}_{\mu ,\lambda })=\infty .  \label{4h}
\end{equation}%
Thus, (\ref{4i}) holds by (\ref{4j}) and (\ref{4h}).

The proof is complete.
\end{proof}

\begin{lemma}
\label{Le6}Consider (\ref{eq1}). Let%
\begin{equation}
\Phi (\alpha ,\lambda )\equiv \int_{0}^{1}\frac{\alpha \lambda E_{t}(\alpha
)+\alpha }{\sqrt{\lambda ^{2}E_{t}^{2}(\alpha )+2\lambda E_{t}(\alpha )}}dt%
\text{ \ for }0<\alpha <c^{\ast }\text{ and }\lambda >0,  \label{PH}
\end{equation}%
where $E_{t}$ is defined by (\ref{ET}). Then the following statements
(i)--(iii) hold.

\begin{itemize}
\item[(i)] If $T_{\mu ,\lambda }(\theta _{\mu ,\lambda })=L$ for some $%
\left( \mu ,\lambda \right) \in \Omega $, then $\Phi (\theta _{\mu ,\lambda
},\lambda )=L$.

\item[(ii)] $\Phi (\alpha ,\lambda )>L$ if $L\leq \alpha <c^{\ast }$ and $%
\lambda >0.$

\item[(iii)] For any $\alpha \in (0,c_{L}^{\ast })$, there exists unique $%
\hat{\lambda}=\hat{\lambda}(\alpha )>0$ such that $\Phi (\alpha ,\hat{\lambda%
}(\alpha ))=L$ where $c_{L}^{\ast }$ is defined in Section 2. Furthermore,

\begin{itemize}
\item[(a)] $\alpha =\theta _{\hat{\mu},\hat{\lambda}}$ and $T_{\hat{\mu},%
\hat{\lambda}}(\theta _{\hat{\mu},\hat{\lambda}})=L$ for any $\alpha \in
(0,c_{L}^{\ast })$ where%
\begin{equation}
\hat{\mu}=\hat{\mu}(\alpha )\equiv \hat{\lambda}(\alpha )\frac{G(\alpha )}{%
\alpha }<\mu _{\hat{\lambda}}.  \label{Le5c}
\end{equation}

\item[(b)] Both $\hat{\lambda}$ and $\hat{\mu}$ are continuously
differentiable and strictly increasing functions on $(0,c_{L}^{\ast })$.

\item[(c)] $\hat{\lambda}(0^{+})=4\kappa $ and $\hat{\lambda}((c_{L}^{\ast
})^{-})=\infty .$

\item[(d)] $\hat{\mu}(0^{+})=0$ and $\hat{\mu}((c_{L}^{\ast })^{-})=\infty .$

\item[(e)] $\hat{\lambda}\searrow 4\kappa $ as $\hat{\mu}\searrow 0$, and $%
\hat{\lambda}\nearrow \infty $ as $\hat{\mu}\nearrow \infty $.

\item[(f)] $\hat{\mu}\searrow 0$ as $\hat{\lambda}\searrow 4\kappa $, and $%
\hat{\mu}\nearrow \infty $ as $\hat{\lambda}\nearrow \infty $.
\end{itemize}
\end{itemize}
\end{lemma}

\begin{proof}
We divide this proof into the following seven steps.

\smallskip \textbf{Step 1.} We prove the statements (i) and (ii). By (\ref{T}%
) and (\ref{t2}), we have
\begin{equation}
T_{\mu ,\lambda }(\theta _{\mu ,\lambda })=\int_{0}^{1}\frac{\theta _{\mu
,\lambda }\left[ \lambda E_{t}(\theta _{\mu ,\lambda })+1\right] }{\sqrt{%
\lambda ^{2}E_{t}^{2}(\theta _{\mu ,\lambda })+2\lambda E_{t}(\theta _{\mu
,\lambda })}}dt=\Phi (\theta _{\mu ,\lambda },\lambda )\text{ \ for }\left(
\mu ,\lambda \right) \in \Omega .  \label{6a}
\end{equation}%
Assume that $T_{\mu ,\lambda }(\theta _{\mu ,\lambda })=L$ for some $\left(
\mu ,\lambda \right) \in \Omega $. It follows that $\Phi (\theta _{\mu
,\lambda },\lambda )=L$ by (\ref{6a}). Thus the statement (i) holds. If $%
L\leq \alpha <c^{\ast }$, we see that%
\begin{equation*}
\Phi (\alpha ,\lambda )=\int_{0}^{1}\frac{\alpha \left[ \lambda E_{t}(\alpha
)+1\right] }{\sqrt{\left[ \lambda E_{t}(\alpha )+1\right] ^{2}-1}}%
dt>\int_{0}^{1}\alpha dt=\alpha \geq L\text{ \ for }\lambda >0,
\end{equation*}%
which implies that the statement (ii) holds.

\textbf{\smallskip Step 2.} We prove that, for any $\alpha \in
(0,c_{L}^{\ast })$, there exists unique $\hat{\lambda}=\hat{\lambda}(\alpha
)>0$ such that $\Phi (\alpha ,\hat{\lambda}(\alpha ))=L$. By Lemma \ref{LH},
we see that%
\begin{equation}
E_{t}(\alpha )=\alpha t\left[ \frac{G(\alpha )}{\alpha }-\frac{G(\alpha t)}{%
\alpha t}\right] >0\text{ \ for }0<\alpha <c^{\ast }\text{ and }0<t<1,
\label{Lb2}
\end{equation}%
from which it follows that%
\begin{equation}
\frac{\partial }{\partial \lambda }\Phi (\alpha ,\lambda )=-\int_{0}^{1}%
\frac{\alpha E_{t}(\alpha )}{\left[ \lambda ^{2}E_{t}^{2}(\alpha )+2\lambda
E_{t}(\alpha )\right] ^{3/2}}dt<0\text{ \ for }0<\alpha <c^{\ast }\text{ and
}\lambda >0\text{.}  \label{Le5h}
\end{equation}%
Since%
\begin{equation*}
\lim_{\lambda \rightarrow \infty }\Phi (\alpha ,\lambda )=\alpha
<L<\lim_{\lambda \rightarrow 0^{+}}\Phi (\alpha ,\lambda )=\infty \text{ \
for }0<\alpha <c_{L}^{\ast }=\min \{c^{\ast },L\},
\end{equation*}%
and by (\ref{Le5h}), there exists unique $\hat{\lambda}=\hat{\lambda}(\alpha
)>0$ such that $\Phi (\alpha ,\hat{\lambda}(\alpha ))=L$.

\textbf{\smallskip Step 3.} We prove the statement (iii)(a). Let $\alpha \in
(0,c_{L}^{\ast })$ be given. From Step 2, we obtain $\hat{\lambda}$. Then by
Lemma \ref{LH}(ii), we observe that%
\begin{equation*}
\hat{\mu}=\hat{\lambda}\frac{G(\alpha )}{\alpha }<\hat{\lambda}\frac{%
G(c^{\ast })}{c^{\ast }}=\mu _{\hat{\lambda}}\text{ \ and \ }F_{\hat{\mu},%
\hat{\lambda}}(\alpha )=\hat{\lambda}G(\alpha )-\hat{\mu}\alpha =0,
\end{equation*}%
which implies that (\ref{Le5c}) holds and $\alpha =\theta _{\hat{\mu},\hat{%
\lambda}}$. So by (\ref{6a}) and Step 2, we have
\begin{equation*}
T_{\hat{\mu},\hat{\lambda}}(\theta _{\hat{\mu},\hat{\lambda}})=\Phi (\theta
_{\hat{\mu},\hat{\lambda}},\hat{\lambda})=\Phi (\alpha ,\hat{\lambda})=L.
\end{equation*}%
Thus the statement (iii)(a) holds.

\textbf{\smallskip Step 4.} We prove the statement (iii)(b). By Step 2, (\ref%
{Le5h}) and implicit function theorem, we see that $\hat{\lambda}(\alpha )$
is a continuously differentiable function on $(0,c_{L}^{\ast })$. So by (\ref%
{Le5c}), $\hat{\mu}(\alpha )$ is a continuously differentiable function on $%
(0,c_{L}^{\ast })$. We assert that%
\begin{equation}
\frac{\partial }{\partial \alpha }\Phi (\alpha ,\lambda )=\int_{0}^{1}\frac{%
\lambda ^{3}E_{t}^{3}(\alpha )+3\lambda ^{2}E_{t}^{2}(\alpha )+\lambda \left[
2E_{t}(\alpha )-\alpha E_{t}^{\prime }(\alpha )\right] }{\left[ \lambda
^{2}E_{t}^{2}(\alpha )+2\lambda E_{t}(\alpha )\right] ^{3/2}}dt>0
\label{Le5g}
\end{equation}%
for $0<\alpha <c^{\ast }$ and $\lambda >0$. Since%
\begin{equation*}
0=\frac{\partial }{\partial \alpha }L=\frac{\partial }{\partial \alpha }\Phi
(\alpha ,\hat{\lambda}(\alpha ))=\left[ \frac{\partial }{\partial \alpha }%
\Phi (\alpha ,\lambda )\right] _{\lambda =\hat{\lambda}(\alpha )}+\left[
\frac{\partial }{\partial \lambda }\Phi (\alpha ,\lambda )\right] _{\lambda =%
\hat{\lambda}(\alpha )}\hat{\lambda}^{\prime }(\alpha ),
\end{equation*}%
for $0<\alpha <c_{L}^{\ast }$, and by (\ref{Le5h}) and (\ref{Le5g}), we see
that%
\begin{equation}
\hat{\lambda}^{\prime }(\alpha )=-\frac{\left[ \frac{\partial }{\partial
\alpha }\Phi (\alpha ,\lambda )\right] _{\lambda =\hat{\lambda}(\alpha )}}{%
\left[ \frac{\partial }{\partial \lambda }\Phi (\alpha ,\lambda )\right]
_{\lambda =\hat{\lambda}(\alpha )}}>0\text{ \ for }0<\alpha <c_{L}^{\ast }.
\label{I1}
\end{equation}%
By (\ref{I1}) and Lemma \ref{LH}(ii), we further see that%
\begin{equation*}
\hat{\mu}^{\prime }(\alpha )=\hat{\lambda}^{\prime }(\alpha )\frac{G(\alpha )%
}{\alpha }+\hat{\lambda}(\alpha )\left[ \frac{G(\alpha )}{\alpha }\right]
^{\prime }>0\text{ \ for }0<\alpha <c_{L}^{\ast }.
\end{equation*}%
Based on the above discussions, the statement (iii)(b) holds.

Next, we prove the assertion (\ref{Le5g}). By (\ref{Lb2}) and (\ref{Le5g}),
it is sufficient to prove that
\begin{equation}
2E_{t}(\alpha )-\alpha E_{t}^{\prime }(\alpha )>0\text{ \ for }0<\alpha
<c^{\ast }\text{ and }0<t<1.  \label{Lb9}
\end{equation}%
We compute%
\begin{equation}
2E_{t}(\alpha )-\alpha E_{t}^{\prime }(\alpha )=\Lambda (\alpha )t-\Lambda
(\alpha t)=\alpha t\left[ \frac{\Lambda (\alpha )}{\alpha }-\frac{\Lambda
(\alpha t)}{\alpha t}\right] ,  \label{Lb3}
\end{equation}%
where $\Lambda (u)\equiv 2G(u)-ug(u)$. We further compute%
\begin{equation}
\frac{\partial }{\partial u}\frac{\Lambda (u)}{u}=\frac{u\Lambda ^{\prime
}(u)-\Lambda (u)}{u^{2}}.  \label{Lb5}
\end{equation}%
Since $g^{\prime \prime }(u)<0$ on $\left( 0,\sigma \right) $, we see that%
\begin{equation}
\frac{\partial }{\partial u}\left[ u\Lambda ^{\prime }(u)-\Lambda (u)\right]
=u\Lambda ^{\prime \prime }(u)=-u^{2}g^{\prime \prime }(u)>0\text{ \ for }%
0<u<c^{\ast }.  \label{Lb6}
\end{equation}%
Since $g(0)=G(0)=0$, and by (\ref{gg}), (\ref{Lb5}) and (\ref{Lb6}), we
observe that, for $0<u<c^{\ast }$,
\begin{eqnarray*}
\frac{\partial }{\partial u}\frac{\Lambda (u)}{u} &=&\frac{u\Lambda ^{\prime
}(u)-\Lambda (u)}{u^{2}}>\frac{\lim_{u\rightarrow 0^{+}}\left[ u\Lambda
^{\prime }(u)-\Lambda (u)\right] }{u^{2}}\smallskip \\
&=&\frac{\lim_{u\rightarrow 0^{+}}\left[ 2ug(u)-u^{2}g^{\prime }(u)-2G(u)%
\right] }{u^{2}}=0.
\end{eqnarray*}%
So (\ref{Lb9}) holds by (\ref{Lb3}). It implies that the assertion (\ref%
{Le5g}) holds.

\textbf{\smallskip Step 5.} We prove that $\hat{\lambda}(0^{+})=4\kappa $.
By (\ref{I1}), we have $0\leq \hat{\lambda}(0^{+})<\hat{\lambda}%
((c_{L}^{\ast })^{-})\leq \infty $. Since $E_{t}(0)=0$ for $0<t<1$, and by
L'H\^{o}pital's rule, we see that%
\begin{equation}
\lim_{\alpha \rightarrow 0^{+}}\frac{E_{t}(\alpha )}{\alpha }=\lim_{\alpha
\rightarrow 0^{+}}t\left[ g(\alpha )-g(\alpha t)\right] =0  \label{Lb10}
\end{equation}%
and%
\begin{equation}
\lim_{\alpha \rightarrow 0^{+}}\frac{E_{t}(\alpha )}{\alpha ^{2}}%
=\lim_{\alpha \rightarrow 0^{+}}\frac{t\left[ g^{\prime }(\alpha
)-tg^{\prime }(\alpha t)\right] }{2}=\frac{g^{\prime }(0^{+})}{2}\left(
t-t^{2}\right)  \label{Lb13}
\end{equation}%
for $0<t<1$. Notice that%
\begin{equation}
L=\Phi (\alpha ,\hat{\lambda}(\alpha ))=\int_{0}^{1}\frac{\hat{\lambda}%
(\alpha )E_{t}(\alpha )+1}{\sqrt{\hat{\lambda}^{2}(\alpha )\left[ \frac{%
E_{t}(\alpha )}{\alpha }\right] ^{2}+2\hat{\lambda}(\alpha )\frac{%
E_{t}(\alpha )}{\alpha ^{2}}}}dt  \label{Lb12}
\end{equation}%
for $0<\alpha <c^{\ast }$. Then we consider two cases.

Case 1. Assume that $g^{\prime }(0^{+})=\infty $. Suppose $\hat{\lambda}%
(0^{+})>0$. By (\ref{Lb10})--(\ref{Lb12}), we observe that%
\begin{equation*}
L=\lim_{\alpha \rightarrow 0^{+}}\Phi (\alpha ,\hat{\lambda}(\alpha ))=0.
\end{equation*}%
which is a contradiction. Thus $\hat{\lambda}(0^{+})=0=4\kappa $.

Case 2. Assume that $g^{\prime }(0^{+})\in (0,\infty )$. Suppose $\hat{%
\lambda}(0^{+})=0$. By (\ref{Lb10})--(\ref{Lb12}), we observe that%
\begin{equation*}
L=\lim_{\alpha \rightarrow 0^{+}}\Phi (\alpha ,\hat{\lambda}(\alpha
))=\infty .
\end{equation*}%
which is a contradiction. Thus $\hat{\lambda}(0^{+})>0$. Again, by (\ref%
{Lb10})--(\ref{Lb12}), then%
\begin{equation*}
L=\lim_{\alpha \rightarrow 0^{+}}\Phi (\alpha ,\hat{\lambda}(\alpha ))=\frac{%
1}{\sqrt{\hat{\lambda}(0^{+})g^{\prime }(0^{+})}}\int_{0}^{1}\frac{1}{\sqrt{%
t-t^{2}}}dt=\frac{\pi }{\sqrt{\hat{\lambda}(0^{+})g^{\prime }(0^{+})}},
\end{equation*}%
from which it follows that
\begin{equation*}
\hat{\lambda}(0^{+})=\frac{\pi ^{2}}{g^{\prime }(0^{+})L^{2}}=4\kappa .
\end{equation*}

\smallskip \textbf{Step 6. }We prove the statement (iii)(c). By Step 5, it
is sufficient to prove that $\hat{\lambda}((c_{L}^{\ast })^{-})=\infty $.
Suppose $\hat{\lambda}((c_{L}^{\ast })^{-})<\infty $. Let $\lambda _{1}=\hat{%
\lambda}((c_{L}^{\ast })^{-})$. Then we consider two cases.

\smallskip Case 1. Assume that $c^{\ast }\leq L$. It implies that $%
c_{L}^{\ast }=c^{\ast }$. By (\ref{4d}), there exist $\bar{M}>0$ and $\bar{%
\delta}\in (0,1)$ such that
\begin{equation}
0<E_{t}(c^{\ast })<\bar{M}(1-t)^{2}<1\text{ \ for }\bar{\delta}<t<1.
\label{Lb14}
\end{equation}%
By (\ref{Lb14}), we observe that%
\begin{eqnarray*}
L &=&\lim_{\alpha \rightarrow (c^{\ast })^{-}}\Phi (\alpha ,\hat{\lambda}%
(\alpha ))=\int_{0}^{1}\frac{c^{\ast }\left[ \lambda _{1}E_{t}(c^{\ast })+1%
\right] }{\sqrt{\lambda _{1}^{2}E_{t}^{2}(c^{\ast })+2\lambda
_{1}E_{t}(c^{\ast })}}dt \\
&\geq &\int_{\bar{\delta}}^{1}\frac{c^{\ast }}{\sqrt{\lambda
_{1}^{2}E_{t}^{2}(c^{\ast })+2\lambda _{1}E_{t}(c^{\ast })}}dt\geq \frac{%
c^{\ast }}{\sqrt{\lambda _{1}^{2}+2\lambda _{1}}}\int_{\bar{\delta}}^{1}%
\frac{1}{\sqrt{E_{t}(c^{\ast })}}dt \\
&=&\frac{c^{\ast }}{\sqrt{\left( \lambda _{1}^{2}+2\lambda _{1}\right) \bar{M%
}}}\int_{\bar{\delta}}^{1}\frac{1}{1-t}dt=\infty ,
\end{eqnarray*}%
which is a contradiction.

\smallskip Case 2. Assume that $c^{\ast }>L$. It implies that $c_{L}^{\ast
}=L$. Let $\hat{E}_{t}=E_{t}(c_{L}^{\ast })\in (0,\infty )$. Then we observe
that%
\begin{equation*}
L=\lim_{\alpha \rightarrow (c_{L}^{\ast })^{-}}\Phi (\alpha ,\hat{\lambda}%
(\alpha ))=c_{L}^{\ast }\int_{0}^{1}\frac{\lambda _{2}\hat{E}_{t}+1}{\sqrt{%
\left[ \lambda _{2}\hat{E}_{t}+1\right] ^{2}-1}}dt>c_{L}^{\ast }=L,
\end{equation*}%
which is a contradiction.

\noindent By Cases 1--2, we obtain $\hat{\lambda}((c_{L}^{\ast
})^{-})=\infty $. The statement (iii)(c) holds.

\smallskip \textbf{Step 7}. We prove the statement (iii)(d)--(f). By L'H\^{o}%
pital's rule, we observe that%
\begin{equation*}
\lim_{\alpha \rightarrow 0^{+}}\frac{G(\alpha )}{\alpha }=\lim_{\alpha
\rightarrow 0^{+}}g(\alpha )=0\text{ \ and \ }\lim_{\alpha \rightarrow
(c_{L}^{\ast })^{-}}\frac{G(\alpha )}{\alpha }=\frac{G(c_{L}^{\ast })}{%
c_{L}^{\ast }}\in (0,\infty ).
\end{equation*}%
So by (\ref{Le5c}), then%
\begin{equation*}
\lim_{\alpha \rightarrow 0^{+}}\hat{\mu}(\alpha )=\lim_{\alpha \rightarrow
0^{+}}\hat{\lambda}(\alpha )\frac{G(\alpha )}{\alpha }=0\text{ \ and \ }%
\lim_{\alpha \rightarrow \left( c_{L}^{\ast }\right) ^{-}}\hat{\mu}(\alpha
)=\lim_{\alpha \rightarrow \left( c_{L}^{\ast }\right) ^{-}}\hat{\lambda}%
(\alpha )\frac{G(\alpha )}{\alpha }=\infty .
\end{equation*}%
Thus, the statement (iii)(d) holds. Finally, the statements (iii)(e) and
(iii)(f) hold by the statements (iii)(b)--(iii)(d).

The proof is complete.
\end{proof}

\begin{lemma}
\label{Le7}Consider (\ref{eq1}). For any $\mu >0$, there exist $\bar{\lambda}%
\in \left( \lambda _{\mu },\infty \right) $ and $\lambda ^{\ast }\in \left(
\lambda _{\mu },\bar{\lambda}\right) $ such that%
\begin{equation}
T_{\mu ,\lambda }(\theta _{\mu ,\lambda })\left\{
\begin{array}{ll}
>L & \text{for }\lambda _{\mu }<\lambda <\bar{\lambda}, \\
=L & \text{for }\lambda =\bar{\lambda}, \\
<L & \text{for }\lambda >\bar{\lambda},%
\end{array}%
\right. \text{ \ and \ }T_{\mu ,\lambda }(\tilde{\alpha}_{\mu ,\lambda
})\left\{
\begin{array}{ll}
>L & \text{for }\lambda _{\mu }<\lambda <\lambda ^{\ast }, \\
=L & \text{for }\lambda =\lambda ^{\ast }, \\
<L & \text{for }\lambda >\lambda ^{\ast }.%
\end{array}%
\right.  \label{7b}
\end{equation}%
Furthermore, both $\bar{\lambda}=\bar{\lambda}(\mu )$ and $\lambda ^{\ast
}=\lambda ^{\ast }(\mu )$ are strictly increasing and continuous functions
on $\left( 0,\infty \right) $,%
\begin{equation}
\lim_{\mu \rightarrow 0^{+}}\bar{\lambda}(\mu )=4\kappa \text{, \ }\lim_{\mu
\rightarrow 0^{+}}\lambda ^{\ast }(\mu )=\kappa \text{ \ and \ }\lim_{\mu
\rightarrow \infty }\bar{\lambda}(\mu )=\lim_{\mu \rightarrow \infty
}\lambda ^{\ast }(\mu )=\infty .  \label{7a}
\end{equation}
\end{lemma}

\begin{proof}
By Lemma \ref{Le6}(iii), the inverse function $\hat{\mu}^{-1}:\left(
0,\infty \right) \rightarrow \left( 0,c_{L}^{\ast }\right) $ exists, and it
is strictly increasing and continuous on $\left( 0,\infty \right) $. Let $%
\bar{\lambda}=\bar{\lambda}(\mu )\equiv \hat{\lambda}\circ \hat{\mu}%
^{-1}(\mu )$. By Lemma \ref{Le6}(iii) again, we see that $\bar{\lambda}$ is
a strictly increasing and continuous function on $\left( 0,\infty \right) $,%
\begin{equation}
\lim_{\mu \rightarrow 0^{+}}\bar{\lambda}(\mu )=4\kappa \text{ \ and \ }%
\lim_{\mu \rightarrow \infty }\bar{\lambda}(\mu )=\infty .  \label{77b}
\end{equation}%
Let $\mu >0$ be given. By Lemma \ref{Le6}(iii)(d), there exists $\alpha
_{1}\in (0,c_{L}^{\ast })$ such that $\mu =\hat{\mu}(\alpha _{1})$. Clearly,
\begin{equation*}
\bar{\lambda}=\bar{\lambda}(\mu )=\hat{\lambda}\circ \hat{\mu}^{-1}(\mu )=%
\hat{\lambda}(\alpha _{1}).
\end{equation*}%
By Lemmas \ref{LH} and \ref{Le6}(iii), we see that
\begin{equation*}
T_{\mu ,\bar{\lambda}}(\theta _{\mu ,\bar{\lambda}})=L\text{ \ and \ }\bar{%
\lambda}=\frac{\mu \alpha _{1}}{G(\alpha _{1})}>\frac{\mu c^{\ast }}{%
G(c^{\ast })}=\lambda _{\mu }.
\end{equation*}%
Suppose there exists $\lambda _{1}\in (\lambda _{\mu },\infty )\backslash \{%
\bar{\lambda}\}$ such that $T_{\mu ,\lambda _{1}}(\theta _{\mu ,\lambda
_{1}})=L$. By Lemma \ref{Le6}, we obtain
\begin{equation*}
\Phi (\theta _{\mu ,\lambda _{1}},\lambda _{1})=L\text{, \ }\lambda _{1}=%
\hat{\lambda}(\theta _{\mu ,\lambda _{1}})\text{ \ and \ }\mu =\hat{\mu}%
(\theta _{\mu ,\lambda _{1}}),
\end{equation*}%
from which it follows that $\lambda _{1}=\hat{\lambda}(\hat{\mu}^{-1}(\mu ))=%
\bar{\lambda}(\mu )=\bar{\lambda}$. It is a contradiction. Thus, by Lemma %
\ref{Le4}(i)(a) and continuity of $T_{\mu ,\lambda }(\theta _{\mu ,\lambda
}) $ with respect to $\lambda $, we obtain%
\begin{equation*}
T_{\mu ,\lambda }(\theta _{\mu ,\lambda })\left\{
\begin{array}{ll}
>L & \text{for }\lambda _{\mu }<\lambda <\bar{\lambda}, \\
=L & \text{for }\lambda =\bar{\lambda}, \\
<L & \text{for }\lambda >\bar{\lambda}.%
\end{array}%
\right.
\end{equation*}%
In addition, by Lemma \ref{Le4}(i)(b), there exists $\lambda ^{\ast
}=\lambda ^{\ast }(\mu )\in (\lambda _{\mu },\infty )$ such that
\begin{equation}
T_{\mu ,\lambda }(\tilde{\alpha}_{\mu ,\lambda })\left\{
\begin{array}{ll}
>L & \text{for }\lambda _{\mu }<\lambda <\lambda ^{\ast }, \\
=L & \text{for }\lambda =\lambda ^{\ast }, \\
<L & \text{for }\lambda >\lambda ^{\ast }.%
\end{array}%
\right.  \label{77d}
\end{equation}%
Let $\Psi (\mu ,\lambda )\equiv T_{\mu ,\lambda }(\tilde{\alpha}_{\mu
,\lambda })-L$. By (\ref{Le4e}) and (\ref{77d}), we have%
\begin{equation*}
\Psi (\mu ,\lambda ^{\ast })=0\text{ \ and \ }\left[ \frac{\partial }{%
\partial \lambda }\Psi (\mu ,\lambda )\right] _{\lambda =\lambda ^{\ast }}<0%
\text{ \ for }\mu >0.
\end{equation*}%
So by implicit function theorem, $\lambda ^{\ast }=\lambda ^{\ast }(\mu )$
is a continuously differentiable function on $\left( 0,\infty \right) $.
Moreover,
\begin{eqnarray*}
0 &=&\frac{\partial }{\partial \mu }\Psi (\mu ,\lambda ^{\ast }(\mu )) \\
&=&\left[ \frac{\partial }{\partial \mu }T_{\mu ,\lambda }(\alpha )\right]
_{\lambda =\lambda ^{\ast }\text{, }\alpha =\tilde{\alpha}_{\mu ,\lambda
^{\ast }}}+\left[ \frac{\partial }{\partial \lambda }T_{\mu ,\lambda
}(\alpha )\right] _{\lambda =\lambda ^{\ast }\text{, }\alpha =\tilde{\alpha}%
_{\mu ,\lambda ^{\ast }}}\lambda ^{\ast \prime }(\mu ).
\end{eqnarray*}%
So by (\ref{Le4b}) and (\ref{Le4a}), we obtain that%
\begin{equation}
\lambda ^{\ast \prime }(\mu )=-\frac{\left[ \frac{\partial }{\partial \mu }%
T_{\mu ,\lambda }(\alpha )\right] _{\lambda =\lambda ^{\ast }\text{, }\alpha
=\tilde{\alpha}_{\mu ,\lambda ^{\ast }}}}{\left[ \frac{\partial }{\partial
\lambda }T_{\mu ,\lambda }(\alpha )\right] _{\lambda =\lambda ^{\ast }\text{%
, }\alpha =\tilde{\alpha}_{\mu ,\lambda ^{\ast }}}}>0\text{ \ for }\mu >0.
\label{L8}
\end{equation}%
Since
\begin{equation*}
\lim_{\mu \rightarrow \infty }\lambda ^{\ast }(\mu )\geq \lim_{\mu
\rightarrow \infty }\lambda _{\mu }=\lim_{\mu \rightarrow \infty }\frac{\mu
}{g(u_{0})}=\infty ,
\end{equation*}%
we see that
\begin{equation*}
\lim_{\mu \rightarrow \infty }\lambda ^{\ast }(\mu )=\infty .
\end{equation*}%
If $\lambda ^{\ast }(\mu )\geq \bar{\lambda}(\mu )$ for some $\mu >0$, by (%
\ref{7b}), (\ref{77d}) and Lemma \ref{Le3}, then%
\begin{equation*}
T_{\mu ,\lambda ^{\ast }}(\theta _{\mu ,\lambda ^{\ast }})\leq L=T_{\mu
,\lambda ^{\ast }}(\tilde{\alpha}_{\mu ,\lambda ^{\ast }})<T_{\mu ,\lambda
^{\ast }}(\theta _{\mu ,\lambda ^{\ast }}),
\end{equation*}%
which is a contradiction. So
\begin{equation}
\lambda ^{\ast }=\lambda ^{\ast }(\mu )<\bar{\lambda}(\mu )=\bar{\lambda}%
\text{ \ for }\mu >0.  \label{77f}
\end{equation}%
Next, we consider two cases.

\smallskip Case 1. Assume that $g^{\prime }(0^{+})=\infty $. By (\ref{77b})
and (\ref{77f}), then $\lim_{\mu \rightarrow 0^{+}}\lambda ^{\ast }(\mu
)=0=\kappa $.

\smallskip Case 2. Assume that $g^{\prime }(0^{+})\in (0,\infty )$. By (\ref%
{77d}) and Lemma \ref{Le0}, we see that%
\begin{equation*}
L=T_{\mu ,\lambda ^{\ast }(\mu )}(\tilde{\alpha}_{\mu ,\lambda ^{\ast }(\mu
)})>T_{0,\lambda ^{\ast }(\mu )}(\tilde{\alpha}_{\mu ,\lambda ^{\ast }(\mu
)})>\eta =\frac{\pi }{2\sqrt{\lambda ^{\ast }(\mu )g^{\prime }(0^{+})}}\text{
\ for }\mu >0.
\end{equation*}%
It follows that
\begin{equation}
\kappa =\frac{\pi ^{2}}{4g^{\prime }(0^{+})L^{2}}<\lambda ^{\ast }(\mu )%
\text{ \ for }\mu >0.  \label{L8c}
\end{equation}%
Now, let $\lambda _{2}=\lim_{\mu \rightarrow 0^{+}}\lambda ^{\ast }(\mu )$.
By (\ref{L8}) and (\ref{L8c}), then $\kappa \leq \lambda _{2}<\lambda ^{\ast
}(\mu )$ for $\mu >0$. Take $\breve{\alpha}\in \left( 0,c^{\ast }\right) $.
By (\ref{Le4a}), (\ref{77d}), Lemmas \ref{Le0} and \ref{Le3}, we observe that%
\begin{equation}
L=\lim_{\mu \rightarrow 0^{+}}T_{\mu ,\lambda ^{\ast }}(\tilde{\alpha}_{\mu
,\lambda ^{\ast }})\leq \lim_{\mu \rightarrow 0^{+}}T_{\mu ,\lambda ^{\ast
}}(\breve{\alpha})\leq \lim_{\mu \rightarrow 0^{+}}T_{\mu ,\lambda _{2}}(%
\breve{\alpha})=T_{0,\lambda _{2}}(\breve{\alpha}).  \label{L8d}
\end{equation}%
Since $\breve{\alpha}$ is arbitrary, we take $\breve{\alpha}\rightarrow
0^{+} $, and by (\ref{L8d}), we obtain
\begin{equation*}
L\leq \eta =\frac{\pi }{2\sqrt{\lambda _{2}g^{\prime }(0^{+})}},
\end{equation*}%
which implies that
\begin{equation*}
\lambda _{2}\leq \frac{\pi ^{2}}{4g^{\prime }(0^{+})L^{2}}=\kappa .
\end{equation*}%
Thus, $\lim_{\mu \rightarrow 0^{+}}\lambda ^{\ast }(\mu )=\kappa $. The
proof is complete.
\end{proof}

\begin{lemma}
\label{Le8}Consider (\ref{eq1}). Let $\lambda >0$ be given. Then the
following statements (i)--(iv) hold.

\begin{itemize}
\item[(i)] If $0<L\leq \eta $, then $T_{\mu ,\lambda }(\alpha )>L$ for $%
\theta _{\mu ,\lambda }\leq \alpha <\beta _{\mu ,\lambda }$ and $\mu \in
(0,\mu _{\lambda }).$

\item[(ii)] If $L>\eta $, then there exists $\mu ^{\ast }\in (0,\mu
_{\lambda })$ such that%
\begin{equation}
T_{\mu ,\lambda }(\tilde{\alpha}_{\mu ,\lambda })\left\{
\begin{array}{ll}
<L & \text{for }0<\mu <\mu ^{\ast }, \\
=L & \text{for }\mu =\mu ^{\ast }, \\
>L & \text{for }\mu ^{\ast }<\mu <\mu _{\lambda }\text{.}%
\end{array}%
\right.  \label{7d}
\end{equation}%
Moreover, $\lim_{\mu \rightarrow 0^{+}}\tilde{\alpha}_{\mu ,\lambda }=0.$

\item[(iii)] If $\eta <L\leq 2\eta $, then $T_{\mu ,\lambda }(\theta _{\mu
,\lambda })>L$ for $\mu \in (0,\mu _{\lambda })$.

\item[(iv)] If $L>2\eta $, then there exists $\bar{\mu}\in (0,\mu ^{\ast })$
such that%
\begin{equation}
T_{\mu ,\lambda }(\theta _{\mu ,\lambda })\left\{
\begin{array}{ll}
<L & \text{for }0<\mu <\bar{\mu}, \\
=L & \text{for }\mu =\bar{\mu}, \\
>L & \text{for }\bar{\mu}<\mu <\mu _{\lambda }.%
\end{array}%
\right.  \label{7g}
\end{equation}%
Furthermore, $\bar{\mu}=\bar{\mu}(\lambda )$ is a strictly increasing and
continuous function on $\left( 4\kappa ,\infty \right) $,%
\begin{equation}
\lim_{\lambda \rightarrow \left( 4\kappa \right) ^{+}}\bar{\mu}(\lambda )=0%
\text{ \ and \ }\lim_{\lambda \rightarrow \infty }\bar{\mu}(\lambda )=\infty
\text{.}  \label{7c}
\end{equation}
\end{itemize}
\end{lemma}

\begin{proof}
(I) If $0<L\leq \eta $, by Lemmas \ref{Le3} and \ref{Le4}(ii)(b), then%
\begin{equation*}
L\leq \eta =\lim_{\mu \rightarrow 0^{+}}T_{\mu ,\lambda }(\tilde{\alpha}%
_{\mu ,\lambda })<T_{\mu ,\lambda }(\tilde{\alpha}_{\mu ,\lambda })\leq
T_{\mu ,\lambda }(\alpha )\text{ }
\end{equation*}%
\ for $\theta _{\mu ,\lambda }\leq \alpha <\beta _{\mu ,\lambda }$ and $\mu
\in (0,\mu _{\lambda })$. The statement (i) holds.

(II) Let $L>\eta $. By Lemma \ref{Le4}(ii)(b), there exists $\mu ^{\ast }\in
(0,\mu _{\lambda })$ such that (\ref{7d}) holds. In addition, by Lemmas \ref%
{Le0} and \ref{Le4}, we see that%
\begin{equation*}
\eta \leq \lim_{\mu \rightarrow 0^{+}}T_{0,\lambda }(\tilde{\alpha}_{\mu
,\lambda })\leq \lim_{\mu \rightarrow 0^{+}}T_{\mu ,\lambda }(\tilde{\alpha}%
_{\mu ,\lambda })=\eta ,
\end{equation*}%
which implies that $T_{0,\lambda }(\lim_{\mu \rightarrow 0^{+}}\tilde{\alpha}%
_{\mu ,\lambda })=0$. So by Lemma \ref{Le0} again, we obtain $\lim_{\mu
\rightarrow 0^{+}}\tilde{\alpha}_{\mu ,\lambda }=0.$ The statement (ii)
holds.

(III) Let $\eta <L\leq 2\eta $. Suppose that there exists $\mu _{1}\in
(0,\mu _{\lambda })$ such that $T_{\mu _{1},\lambda }(\theta _{\mu
_{1},\lambda })=L$. By Lemma \ref{Le7}, then $\lambda =\bar{\lambda}(\mu
_{1})>4\kappa $. Since $L\leq 2\eta $, we see that
\begin{equation*}
L^{2}\leq 4\eta ^{2}=\frac{\pi ^{2}}{\lambda g^{\prime }(0^{+})}<\frac{\pi
^{2}}{4\kappa g^{\prime }(0^{+})}=L^{2},
\end{equation*}%
which is a contradiction. Thus $T_{\mu ,\lambda }(\theta _{\mu ,\lambda
})\neq L$ for $\mu \in (0,\mu _{\lambda })$. Then by Lemma \ref{Le4}(ii)(a),
we obtain $T_{\mu ,\lambda }(\theta _{\mu ,\lambda })>L$ for $\mu \in (0,\mu
_{\lambda })$. The statement (iii) holds.

(IV) Let $L>2\eta $. By Lemma \ref{Le6}(iii), the inverse function $\hat{%
\lambda}^{-1}:\left( 4\kappa ,\infty \right) \rightarrow \left(
0,c_{L}^{\ast }\right) $ exists, and it is strictly increasing and
continuous on $\left( 4\kappa ,\infty \right) $. Let $\bar{\mu}=\bar{\mu}%
(\lambda )\equiv \hat{\mu}\circ \hat{\lambda}^{-1}(\lambda )$. By Lemma \ref%
{Le6}(iii) again, then $\bar{\mu}$ is a strictly increasing and continuous
function on $\left( 4\kappa ,\infty \right) $, and (\ref{7c}) holds. Since%
\begin{equation*}
L>2\eta =\left\{
\begin{array}{ll}
\frac{\pi }{\sqrt{\lambda g^{\prime }(0^{+})}} & \text{if }g^{\prime
}(0^{+})\in (0,\infty ), \\
0 & \text{if }g^{\prime }(0^{+})=\infty ,%
\end{array}%
\right.
\end{equation*}%
we observe that $\lambda >4\kappa $. Let $\alpha _{1}=\hat{\lambda}%
^{-1}(\lambda )$. It implies that $\lambda =\hat{\lambda}(\alpha _{1})$.
Furthermore,%
\begin{equation*}
\Phi (\alpha _{1},\lambda )=\Phi (\alpha _{1},\hat{\lambda}(\alpha _{1}))=L%
\text{ \ and \ }\bar{\mu}=\bar{\mu}(\lambda )\equiv \hat{\mu}(\alpha _{1}).
\end{equation*}%
So by Lemma \ref{Le6}(iii), then
\begin{equation*}
T_{\bar{\mu},\lambda }(\theta _{\bar{\mu},\lambda })=T_{\hat{\mu},\hat{%
\lambda}}(\theta _{\hat{\mu},\hat{\lambda}})=L\text{ \ and \ }\alpha
_{1}=\theta _{\hat{\mu},\hat{\lambda}}=\theta _{\bar{\mu},\lambda }.
\end{equation*}%
Suppose there exists $\mu _{1}\in (0,\mu _{\lambda })$ such that $T_{\mu
_{1},\lambda }(\theta _{\mu _{1},\lambda })=L$. It follows that $\Phi
(\theta _{\mu _{1},\lambda },\lambda )=L$. Furthermore,
\begin{equation*}
\hat{\lambda}(\theta _{\mu _{1},\lambda })=\lambda =\hat{\lambda}(\alpha
_{1})=\hat{\lambda}(\theta _{\bar{\mu},\lambda }).
\end{equation*}%
Then $\theta _{\bar{\mu},\lambda }=\theta _{\mu _{1},\lambda }$. So by Lemma %
\ref{Le1}(ii), then $\bar{\mu}=\mu _{1}$. Then (\ref{7g}) holds by Lemma \ref%
{Le4}(ii)(a) and continuity of $T_{\mu ,\lambda }(\theta _{\mu ,\lambda })$
respect to $\lambda $. The proof is complete.
\end{proof}

\begin{figure}[h]
\centering
\includegraphics[width=6.3in,height=2.1in,keepaspectratio]{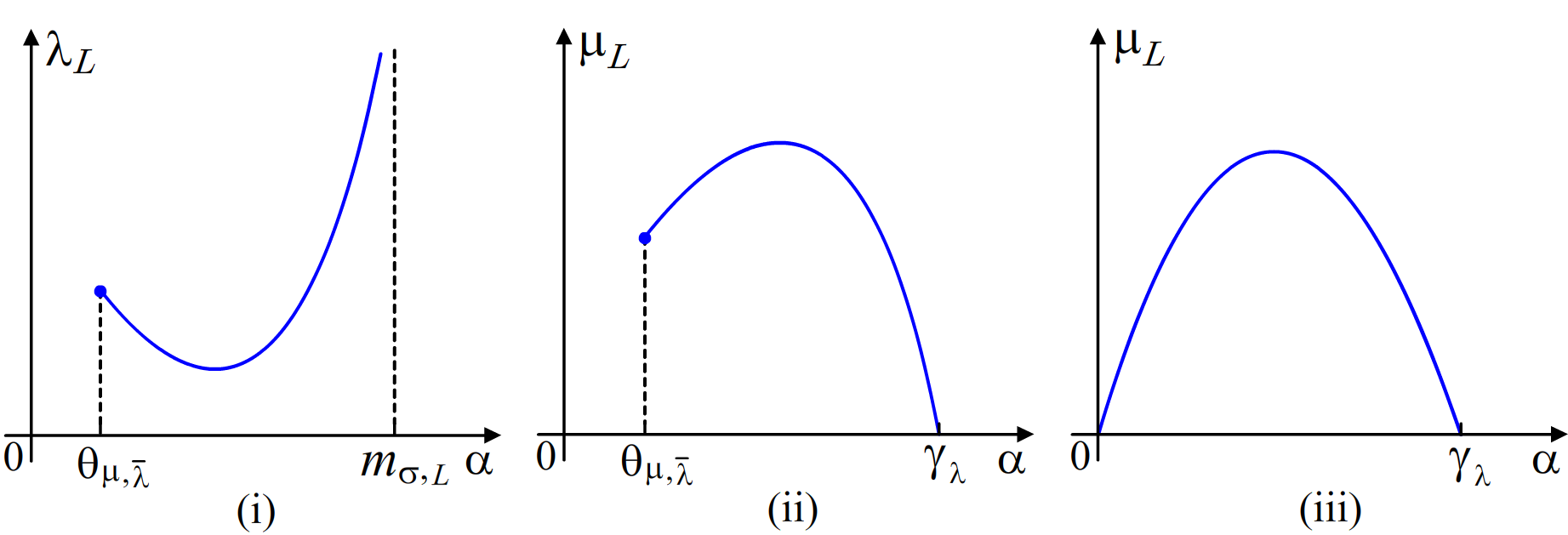}
\caption{i) The graph of $\protect\lambda _{L}(\protect\alpha )$ on $[%
\protect\theta _{\protect\mu ,\bar{\protect\lambda}},m_{\protect\sigma ,L})$%
. (ii) The graph of $\protect\mu _{L}(\protect\alpha )$ on $[\protect\theta %
_{\protect\mu ,\bar{\protect\lambda}},\protect\gamma _{\protect\sigma ,L})$
if $L>2\protect\eta $. (iii) The graph of $\protect\mu _{L}(\protect\alpha )$
on $(0,\protect\gamma _{\protect\sigma ,L})$ if $L\leq 2\protect\eta $.}
\label{fig4}
\end{figure}

\begin{lemma}[See Figure \protect\ref{fig4}(i)]
\label{Le9}Consider (\ref{eq1}). Let $\mu >0$ be given. Then the following
statements (i)--(iv) hold.

\begin{itemize}
\item[(i)] There exists a positive function $\lambda _{L}(\alpha )$ on $%
[\theta _{\mu ,\bar{\lambda}},m_{\sigma ,L})$ such that
\begin{equation}
T_{\mu ,\lambda _{L}(\alpha )}(\alpha )=L\text{ \ and \ }\lambda _{L}(\theta
_{\mu ,\bar{\lambda}})=\bar{\lambda}  \label{p1a}
\end{equation}%
where $\bar{\lambda}$ is defined in Lemma \ref{Le7}.

\item[(ii)] $\lambda _{L}(\alpha )\in C^{1}(\theta _{\mu ,\bar{\lambda}%
},m_{\sigma ,L})$ and
\begin{equation*}
sgn\left( \lambda _{L}^{\prime }(\alpha )\right) =sgn\left( T_{\mu ,\lambda
_{L}(\alpha )}^{\prime }(\alpha )\right) \text{ \ for }\alpha \in \left(
\theta _{\mu ,\bar{\lambda}},m_{\sigma ,L}\right) .
\end{equation*}

\item[(iii)] The bifurcation curve $S_{\mu }=\left\{ \left( \lambda
_{L}(\alpha ),\alpha \right) :\alpha \in \lbrack \theta _{\mu ,\bar{\lambda}%
},m_{\sigma ,L})\right\} $ is continuous on the $\left( \lambda ,\left\Vert
u_{\lambda }\right\Vert _{\infty }\right) $-plane.

\item[(iv)] $\lim_{\alpha \rightarrow m_{\sigma ,L}^{-}}\lambda _{L}(\alpha
)=\infty .$
\end{itemize}
\end{lemma}

\begin{proof}
(I) We consider four cases.

Case 1. $\alpha \in (0,\theta _{\mu ,\bar{\lambda}})$. Suppose there exists $%
\lambda _{1}>\lambda _{\mu }$ such that $T_{\mu ,\lambda _{1}}(\alpha )=L$.
Since $\theta _{\mu ,\lambda _{1}}\leq \alpha <\theta _{\mu ,\bar{\lambda}}$%
, and by Lemma \ref{Le1}, we see that $\lambda _{1}>\bar{\lambda}$ and there
exists $\lambda _{2}\in (\bar{\lambda},\lambda _{1}]$ such that $\alpha
=\theta _{\mu ,\lambda _{2}}$. Then by Lemma \ref{Le7}, we observe that%
\begin{equation*}
T_{\mu ,\lambda _{1}}(\alpha )=L=T_{\mu ,\bar{\lambda}}(\theta _{\mu ,\bar{%
\lambda}})>T_{\mu ,\lambda _{2}}(\theta _{\mu ,\lambda _{2}})=T_{\mu
,\lambda _{2}}(\alpha ),
\end{equation*}%
which is a contradiction by (\ref{Le4a}). Thus $T_{\mu ,\lambda }(\alpha
)\neq L$ for $\lambda >\lambda _{\mu }.$

Case 2. $\alpha \in \lbrack m_{\sigma ,L},\infty )$. Suppose there exists $%
\lambda _{3}>\lambda _{\mu }$ such that $T_{\mu ,\lambda _{3}}(\alpha )=L$.
It follows that%
\begin{equation*}
L=T_{\mu ,\lambda _{3}}(\alpha )>\alpha \geq m_{\sigma ,L}\geq L,
\end{equation*}%
which is a contradiction. Thus $T_{\mu ,\lambda }(\alpha )\neq L$ for $%
\lambda >\lambda _{\mu }.$

Case 3. $\alpha \in \lbrack \theta _{\mu ,\bar{\lambda}},c^{\ast })$. By
Lemma \ref{Le1}, there exists $\lambda _{4}\in (\lambda _{\mu },\bar{\lambda}%
]$ such that $\alpha =\theta _{\mu ,\lambda _{4}}$. By Lemma \ref{Le7}, then
\begin{equation}
T_{\mu ,\lambda _{4}}(\alpha )=T_{\mu ,\lambda _{4}}(\theta _{\mu ,\lambda
_{4}})\geq L\geq m_{\sigma ,L}>\alpha =\lim_{\lambda \rightarrow \infty
}T_{\mu ,\lambda }(\alpha ).  \label{9a}
\end{equation}%
So by (\ref{Le4a}) and continuity of $T_{\mu ,\lambda }(\alpha )$ with
respect to $\lambda $, there exists unique $\lambda _{L}=\lambda _{L}(\alpha
)\geq \lambda _{4}$ such that $T_{\mu ,\lambda _{L}(\alpha )}(\alpha )=L$.
Since%
\begin{equation*}
T_{\mu ,\lambda _{L}(\theta _{\mu ,\bar{\lambda}})}(\theta _{\mu ,\bar{%
\lambda}})=L=T_{\mu ,\bar{\lambda}}(\theta _{\mu ,\bar{\lambda}}),
\end{equation*}%
we observe that $\lambda _{L}(\theta _{\mu ,\bar{\lambda}})=\bar{\lambda}.$

Case 4. $\alpha \in \lbrack c^{\ast },m_{\sigma ,L})$. By Lemmas \ref{Le1}
and \ref{Le2}, there exists $\lambda _{5}\in (\lambda _{\mu },\infty )$ such
that
\begin{equation*}
\alpha \in \lbrack \theta _{\mu ,\lambda _{5}},\beta _{\mu ,\lambda _{5}})%
\text{ \ and \ }T_{\mu ,\lambda _{5}}(\alpha )>L.
\end{equation*}%
By the similar argument in Case 3, there exists unique $\lambda _{L}=\lambda
_{L}(\alpha )\geq \lambda _{5}$ such that $T_{\mu ,\lambda _{L}}(\alpha )=L$
and $\lambda _{L}(\theta _{\mu ,\bar{\lambda}})=\bar{\lambda}.$

\noindent By Cases 1--4, there exists a positive function $\lambda
_{L}(\alpha )$ on $[\theta _{\mu ,\bar{\lambda}},m_{\sigma ,L})$ such that (%
\ref{p1a}) holds.

\smallskip (II) By (\ref{Le4a}) and implicit function theorem, $\lambda
_{L}=\lambda _{L}(\alpha )$ is a continuously differentiable function on $%
\left( \theta _{\mu ,\bar{\lambda}},m_{\sigma ,L}\right) $. Moreover,%
\begin{equation*}
0=\frac{\partial }{\partial \alpha }L=\frac{\partial }{\partial \alpha }%
T_{\mu ,\lambda _{L}(\alpha )}(\alpha )=T_{\mu ,\lambda _{L}(\alpha
)}^{\prime }(\alpha )+\left[ \frac{\partial }{\partial \lambda }T_{\mu
,\lambda }(\alpha )\right] _{\lambda =\lambda _{L}(\alpha )}\lambda
_{L}^{\prime }(\alpha ).
\end{equation*}%
So the statement (ii) holds by (\ref{Le4a}).

\smallskip (III) By Lemmas \ref{Le2} and \ref{Le3}, there exists $\omega \in
\left( \theta _{\mu ,\bar{\lambda}},\beta _{\mu ,\bar{\lambda}}\right) $
such that%
\begin{equation}
T_{\mu ,\bar{\lambda}}(\theta _{\mu ,\bar{\lambda}})=T_{\mu ,\bar{\lambda}%
}(\omega )=L>T_{\mu ,\bar{\lambda}}(\alpha )\text{ \ for }\theta _{\mu ,\bar{%
\lambda}}<\alpha <\omega ,  \label{9c}
\end{equation}%
see Figure \ref{fig7}.
\begin{figure}[h]
\centering
\includegraphics[width=3.2in,height=2.3in,keepaspectratio]{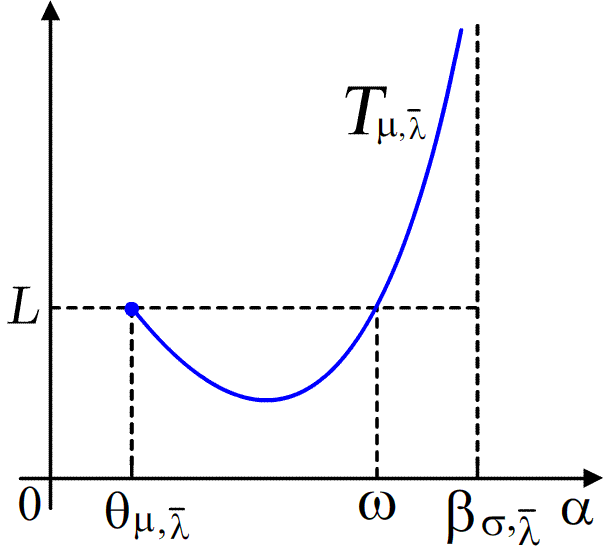}
\caption{The graph of $T_{\protect\mu ,\bar{\protect\lambda}}(\protect\alpha %
)$ on $[\protect\theta _{\protect\mu ,\bar{\protect\lambda}},\protect\beta _{%
\protect\mu ,\bar{\protect\lambda}}).$}
\label{fig7}
\end{figure}
Suppose $\lambda _{L}(\alpha )>\bar{\lambda}$ for some $\alpha \in (\theta
_{\mu ,\bar{\lambda}},\omega )$. By (\ref{Le4a}) and (\ref{9c}), we observe
that%
\begin{equation*}
T_{\mu ,\lambda _{L}(\alpha )}(\alpha )=L=T_{\mu ,\bar{\lambda}}(\theta
_{\mu ,\bar{\lambda}})>T_{\mu ,\bar{\lambda}}(\alpha )>T_{\mu ,\lambda
_{L}(\alpha )}(\alpha ),
\end{equation*}%
which is a contradiction. Thus $\lambda _{L}(\alpha )\leq \bar{\lambda}$ for
$\alpha \in (\theta _{\mu ,\bar{\lambda}},\omega )$. By (\ref{T}) and (\ref%
{p1a}), we have%
\begin{equation}
\theta _{\mu ,\lambda _{L}(\alpha )}\leq \alpha <\beta _{\mu ,\lambda
_{L}(\alpha )}\text{ \ for }\alpha \in (\theta _{\mu ,\bar{\lambda}},\omega
).  \label{9b}
\end{equation}%
By Lemma \ref{Le1} and (\ref{9b}), we obtain%
\begin{equation*}
\theta _{\mu ,\bar{\lambda}}\leq \lim_{\alpha \rightarrow \theta _{\mu ,\bar{%
\lambda}}^{+}}\theta _{\mu ,\lambda _{L}(\alpha )}\leq \lim_{\alpha
\rightarrow \theta _{\mu ,\bar{\lambda}}^{+}}\alpha =\theta _{\mu ,\bar{%
\lambda}}.
\end{equation*}%
Then $\lim_{\alpha \rightarrow \theta _{\mu ,\bar{\lambda}}^{+}}\lambda
_{L}(\alpha )=\bar{\lambda}=\lambda _{L}(\theta _{\mu ,\bar{\lambda}})$,
which implies that $\lambda _{L}(\alpha )$ is continuous on $[\theta _{\mu ,%
\bar{\lambda}},m_{\sigma ,L})$. Moreover, by (\ref{S}),
\begin{equation*}
S_{\mu }=\left\{ \left( \lambda _{L}(\alpha ),\alpha \right) :\alpha \in
\lbrack \theta _{\mu ,\bar{\lambda}},m_{\sigma ,L})\right\}
\end{equation*}%
is continuos. The statement (iii) holds.

\smallskip (IV). By (\ref{T}) and (\ref{p1a}), we have $\lambda _{L}(\alpha
)>\lambda _{\mu }$ on $[\theta _{\mu ,\bar{\lambda}},m_{\sigma ,L})$. Let $%
\lambda _{6}=\liminf_{\alpha \rightarrow m_{\sigma ,L}^{-}}\lambda
_{L}(\alpha )$. Clearly, $\lambda _{6}\in (\lambda _{\mu },\infty ]$.
Suppose $\lambda _{6}<\infty $. We consider two cases:

Case 1. $L\geq \sigma $. Clearly, $m_{\sigma ,L}=\sigma $. By (\ref{9b}) and
Lemma \ref{Le1}, we observe that%
\begin{equation*}
m_{\sigma ,L}=\liminf_{\alpha \rightarrow m_{\sigma ,L}^{-}}\alpha \leq
\liminf_{\alpha \rightarrow m_{\sigma ,L}^{-}}\beta _{\mu ,\lambda
_{L}(\alpha )}=\beta _{\mu ,\lambda _{6}}<\sigma ,
\end{equation*}%
which is a contradiction.

Case 2. $L<\sigma $. Clearly, $m_{\sigma ,L}=L$. By Lemma \ref{Le1}, there
exists $\lambda _{7}\in (\lambda _{6},\infty )$ such that $L\in (\theta
_{\mu ,\lambda _{7}},\beta _{\mu ,\lambda _{7}})$. Since
\begin{equation*}
\liminf_{\alpha \rightarrow L^{-}}\lambda _{L}(\alpha )=\lambda _{6}<\infty ,
\end{equation*}%
there exists a sequence $\{\alpha _{n}\in (\theta _{\mu ,\lambda _{7}},\beta
_{\mu ,\lambda _{7}})\}$ such that $\lim_{n\rightarrow \infty }\alpha _{n}=L$%
, $\alpha _{n}<L$ and $\lambda _{L}(\alpha _{n})<\lambda _{7}$ for $n\in
\mathbb{N}$. Then by (\ref{Le4a}) and (\ref{p1a}), we see that
\begin{equation*}
L<T_{\mu ,\lambda _{7}}(L)=\lim_{n\rightarrow \infty }T_{\mu ,\lambda
_{7}}(\alpha _{n})\leq \lim_{n\rightarrow \infty }T_{\mu ,\lambda
_{L}(\alpha _{n})}(\alpha _{n})=L,
\end{equation*}%
which is a contradiction.

\noindent Thus by Cases 1--2, $\lambda _{6}=\infty $. The statement (iv)
holds. The proof is complete.
\end{proof}

\smallskip

Using a proof similar to that of Lemma \ref{Le9}, we obtain the following
Lemma \ref{Le10}. Due to the length of the proof, the proof is given in the
Appendix.

\smallskip

\begin{lemma}[See Figure \protect\ref{fig4}(ii)(iii)]
\label{Le10}Consider (\ref{eq1}). Let $\lambda >0$ be given. Assume that $%
L>\eta $. Let $\gamma _{\lambda }$ be defined in Lemma \ref{Le0}. Then the
following statements (i)--(v) hold.

\begin{itemize}
\item[(i)] There exists a positive function $\mu _{L}(\alpha )$ on the
interval $I$ such that $T_{\mu _{L}(\alpha ),\lambda }(\alpha )=L$ where $%
\bar{\mu}$ is defined in Lemma \ref{Le8} and%
\begin{equation*}
I\equiv \left\{
\begin{array}{ll}
\lbrack \theta _{\bar{\mu},\lambda },\gamma _{\lambda }) & \text{if }L>2\eta
\text{,}\smallskip \\
(0,\gamma _{\lambda }) & \text{if }\eta <L\leq 2\eta .%
\end{array}%
\right.
\end{equation*}%
Furthermore, $\mu _{L}(\theta _{\bar{\mu},\lambda })=\bar{\mu}$ if $L>2\eta $%
, and $\mu _{L}(0^{+})=0$ if $\eta <L\leq 2\eta $.

\item[(ii)] $\mu _{L}(\alpha )\in C^{1}(\mathring{I})$ and
\begin{equation}
sgn\left( \mu _{L}^{\prime }(\alpha )\right) =-sgn\left( T_{\mu _{L}(\alpha
),\lambda }^{\prime }(\alpha )\right) \text{ \ for }\alpha \in \mathring{I}
\label{10a}
\end{equation}%
where%
\begin{equation*}
\mathring{I}\equiv \left\{
\begin{array}{ll}
\left( \theta _{\bar{\mu},\lambda },\gamma _{\lambda }\right) & \text{if }%
L>2\eta \text{,}\smallskip \\
(0,\gamma _{\lambda }) & \text{if }\eta <L\leq 2\eta .%
\end{array}%
\right.
\end{equation*}

\item[(iii)] The bifurcation curve $\Sigma _{\lambda }=\left\{ \left( \mu
_{L}(\alpha ),\alpha \right) :\alpha \in I\right\} $ is continuous on the $%
\left( \lambda ,\left\Vert u_{\lambda }\right\Vert _{\infty }\right) $-plane.

\item[(iv)] $\lim_{\alpha \rightarrow \gamma _{\lambda }}\mu _{L}(\alpha
)=0. $
\end{itemize}
\end{lemma}

\section{Proofs of Main Results}

In this section, we present the proofs of Theorems \ref{T1}--\ref{T2}.

\begin{proof}[Proof of Theorem \protect\ref{T1}]
(I) By Lemma \ref{Le9}, we see that the bifurcation curve $S_{\mu }$ is
continuous, starts from $(\bar{\lambda},\left\Vert u_{\bar{\lambda}%
}\right\Vert _{\infty })=(\bar{\lambda},\theta _{\mu ,\bar{\lambda}})$ and
goes to $(\infty ,m_{\sigma ,L})$. By Lemmas \ref{Le3} and \ref{Le7}, we
have $T_{\mu ,\lambda ^{\ast }}(\tilde{\alpha}_{\mu ,\lambda ^{\ast }})=L$
and $T_{\mu ,\lambda ^{\ast }}^{\prime }(\tilde{\alpha}_{\mu ,\lambda ^{\ast
}})=0$. So by Lemma \ref{Le9}, we obtain%
\begin{equation}
\lambda ^{\ast }=\lambda _{L}(\tilde{\alpha}_{\mu ,\lambda ^{\ast }})\text{
\ and \ }\lambda _{L}^{\prime }(\tilde{\alpha}_{\mu ,\lambda ^{\ast }})=0.
\label{p1b}
\end{equation}%
Suppose there exists $\alpha _{1},\alpha _{2}\in \lbrack \theta _{\mu ,\bar{%
\lambda}},m_{\sigma ,L})$ such that $\lambda _{L}^{\prime }(\alpha
_{1})=\lambda _{L}^{\prime }(\alpha _{2})=0$. Let $\lambda _{1}=\lambda
_{L}(\alpha _{1})$ and $\lambda _{2}=\lambda _{L}(\alpha _{1})$. By Lemma %
\ref{Le9}, we have%
\begin{equation*}
T_{\mu ,\lambda _{1}}(\alpha _{1})=T_{\mu ,\lambda _{2}}(\alpha _{1})=L\text{
\ and \ }T_{\mu ,\lambda _{1}}^{\prime }(\alpha _{1})=T_{\mu ,\lambda
_{2}}^{\prime }(\alpha _{1})=0,
\end{equation*}%
which, by Lemmas \ref{Le3} and \ref{Le7}, implies that $\alpha _{1}=\tilde{%
\alpha}_{\mu ,\lambda _{1}}=\tilde{\alpha}_{\mu ,\lambda _{2}}=\alpha _{2}$.
Thus by (\ref{p1b}),
\begin{equation}
\lambda _{L}(\alpha )\text{ has exactly one critical number }\tilde{\alpha}%
_{\mu ,\lambda ^{\ast }}\text{\ on }[\theta _{\mu ,\bar{\lambda}},m_{\sigma
,L}).  \label{p1c}
\end{equation}

Let $\lambda _{1}\in (\lambda ^{\ast },\bar{\lambda})$ be given. By Lemma %
\ref{Le7}, we obtain $T_{\mu ,\lambda _{1}}(\tilde{\alpha}_{\mu ,\lambda
_{1}})<L<T_{\mu ,\lambda _{1}}(\theta _{\mu ,\lambda _{1}})$. So by Lemma %
\ref{Le3}, there exists $\alpha _{1}\in \left( \theta _{\mu ,\lambda _{1}},%
\tilde{\alpha}_{\mu ,\lambda _{1}}\right) $ such that
\begin{equation*}
T_{\mu ,\lambda _{1}}(\alpha _{1})=L\text{ \ and \ }T_{\mu ,\lambda
_{1}}^{\prime }(\alpha _{1})<0,
\end{equation*}%
which, by Lemma \ref{Le9}, implies that $\lambda _{1}=\lambda _{L}(\alpha
_{1})$\ and $\lambda _{L}^{\prime }(\alpha _{1})<0$. So by Lemma \ref{Le9}
and (\ref{p1c}), we obtain%
\begin{equation*}
\lambda _{L}^{\prime }(\alpha )\left\{
\begin{array}{ll}
<0 & \text{for }\theta _{\mu ,\bar{\lambda}}<\alpha <\tilde{\alpha}_{\mu
,\lambda ^{\ast }}, \\
=0 & \text{for }\alpha =\tilde{\alpha}_{\mu ,\lambda ^{\ast }}, \\
>0 & \text{for }\tilde{\alpha}_{\mu ,\lambda ^{\ast }}<\alpha <m_{\sigma ,L}.%
\end{array}%
\right.
\end{equation*}%
It implies that $S_{\mu }$ is $\subset $-shape. The statements (i)(a) and
(i)(b) follows by Lemma \ref{Le7}.
\end{proof}

\smallskip

\begin{proof}[Proof of Theorem \protect\ref{T3}]
(I) Assume that $g^{\prime }(0^{+})\in (0,\infty )$. If $0<\lambda \leq
\kappa $, by (\ref{K}) and (\ref{E}), we obtain%
\begin{equation*}
L\leq \frac{\pi }{2\sqrt{g^{\prime }(0^{+})\lambda }}=\eta ,
\end{equation*}%
which, by (\ref{C}) and Lemma \ref{Le8}(i), implies that the bifurcation
curve $\Sigma _{\lambda }$ does not exist. The statement (i)(a) holds.

If $\kappa <\lambda \leq 4\kappa $, by (\ref{K}) and (\ref{E}), we obtain%
\begin{equation*}
\eta =\frac{\pi }{2\sqrt{g^{\prime }(0^{+})\lambda }}<L\leq \frac{\pi }{%
\sqrt{g^{\prime }(0^{+})\lambda }}=2\eta .
\end{equation*}%
So by Lemma \ref{Le10}, the bifurcation curve $\Sigma _{\lambda }$ is
continuous, starts from $(0,0)$ and goes to $(0,\gamma _{\lambda })$. By the
similar argument in the proof of Theorem \ref{T1}, we see that $\mu
_{L}(\alpha )$ has exactly one critical number $\tilde{\alpha}_{\mu ^{\ast
},\lambda }$ on $\left( 0,\gamma _{\lambda }\right) $. Thus
\begin{equation}
\mu _{L}^{\prime }(\alpha )\left\{
\begin{array}{ll}
>0 & \text{for }0<\alpha <\tilde{\alpha}_{\mu ^{\ast },\lambda }, \\
=0 & \text{for }\alpha =\tilde{\alpha}_{\mu ^{\ast },\lambda }, \\
<0 & \text{for }\tilde{\alpha}_{\mu ^{\ast },\lambda }<\alpha <\gamma
_{\lambda }.%
\end{array}%
\right.  \label{t3a}
\end{equation}%
The statement (i)(b) holds.

(II) Assume that $g^{\prime }(0^{+})\in (0,\infty ]$. If $\lambda >4\kappa $%
, by (\ref{K}) and (\ref{E}), we obtain $L>2\eta $. So by Lemma \ref{Le10},
the bifurcation curve $\Sigma _{\lambda }$ is continuous, starts from $%
\left( \bar{\mu},\left\Vert u_{\bar{\mu}}\right\Vert _{\infty }\right) =(%
\bar{\mu},\theta _{\bar{\mu},\lambda })$ and goes to $(0,\gamma _{\lambda })$%
. By the similar argument in the proof of Theorem \ref{T1}, we see that $\mu
_{L}(\alpha )$ has exactly one critical number $\tilde{\alpha}_{\mu ^{\ast
},\lambda }$ on $\left( 0,\gamma _{\lambda }\right) $. Let $\mu _{1}\in (%
\bar{\mu},\mu ^{\ast })$ be given. By Lemma \ref{Le8}, we obtain $T_{\mu
_{1},\lambda }(\tilde{\alpha}_{\mu _{1},\lambda })<L<T_{\mu _{1},\lambda
}(\theta _{\mu _{1},\lambda })$. So by Lemma \ref{Le3}, there exists $\alpha
_{1}\in \left( \theta _{\mu _{1},\lambda },\tilde{\alpha}_{\mu _{1},\lambda
}\right) $ such that
\begin{equation*}
T_{\mu _{1},\lambda }(\alpha _{1})=L\text{ \ and \ }T_{\mu _{1},\lambda
}^{\prime }(\alpha _{1})<0,
\end{equation*}%
which, by Lemma \ref{Le10}, implies that $\mu _{1}=\mu _{L}(\alpha _{1})$\
and $\mu _{L}^{\prime }(\alpha _{1})>0$. So (\ref{t3a}) holds. Finally, (\ref%
{T3A}) holds by (\ref{7c}) and Lemma \ref{Le1}.
\end{proof}

\smallskip

\begin{proof}[Proof of Theorem \protect\ref{T2}]
Theorem \ref{T2} follows by Theorem \ref{T1} and Lemma \ref{Le7}.
\end{proof}

\section{Appendix-the proof of Lemma \protect\ref{Le10}}

We divide the proof of Lemma \ref{Le10} into the following six steps.

\smallskip

\textbf{Step 1.} We prove that $\tilde{\alpha}_{\mu ^{\ast },\lambda
}<\gamma _{\lambda }$. By Lemmas \ref{Le0} and \ref{Le8}, we see that%
\begin{equation*}
T_{0,\lambda }(\gamma _{\lambda })=L=T_{\mu ^{\ast },\lambda }(\tilde{\alpha}%
_{\mu ^{\ast },\lambda })>T_{0,\lambda }(\tilde{\alpha}_{\mu ^{\ast
},\lambda }),
\end{equation*}%
which implies that $\tilde{\alpha}_{\mu ^{\ast },\lambda }<\gamma _{\lambda
} $.

\smallskip

\textbf{Step 2.} We prove the statement (i) if $L>2\eta $. By Step 1, we
consider four cases.

\smallskip Case 1. $\alpha \in (0,\theta _{\bar{\mu},\lambda })$. Suppose
there exists $\mu _{1}\in \left( 0,\mu _{\lambda }\right) $ such that $%
T_{\mu _{1},\lambda }(\alpha )=L$. Since $\theta _{\mu _{1},\lambda }<\alpha
<\theta _{\bar{\mu},\lambda }$, and by Lemma \ref{Le1}, we see that $\mu
_{1}<\bar{\mu}$ and there exists $\mu _{2}\in (\mu _{1},\bar{\mu})$ such
that $\alpha =\theta _{\mu _{2},\lambda }$. Then by Lemma \ref{Le8}, we
observe that%
\begin{equation*}
T_{\mu _{1},\lambda }(\alpha )=L=T_{\bar{\mu},\lambda }(\theta _{\bar{\mu}%
,\lambda })>T_{\mu _{2},\lambda }(\theta _{\mu _{2},\lambda })=T_{\mu
_{2},\lambda }(\alpha ),
\end{equation*}%
which is a contradiction by (\ref{Le4b}). Thus $T_{\mu ,\lambda }(\alpha
)\neq L$ for $\mu \in \left( 0,\mu _{\lambda }\right) .$

\smallskip Case 2. $\alpha \in \lbrack \gamma _{\lambda },\infty )$. Suppose
there exists $\mu _{3}\in \left( 0,\mu _{\lambda }\right) $ such that $%
T_{\mu _{3},\lambda }(\alpha )=L$. By Lemma \ref{Le0}, we see that%
\begin{equation*}
L=T_{\mu _{3},\lambda }(\alpha )>T_{0,\lambda }(\alpha )\geq T_{0,\lambda
}(\gamma _{\lambda })=L,
\end{equation*}%
which is a contradiction. Thus $T_{\mu ,\lambda }(\alpha )\neq L$ for $\mu
\in \left( 0,\mu _{\lambda }\right) .$

\smallskip Case 3. $\alpha \in \lbrack \theta _{\bar{\mu},\lambda },\theta
_{\mu ^{\ast },\lambda }]$. There exists $\mu _{4}\in \lbrack \bar{\mu},\mu
^{\ast }]$ such that $\alpha =\theta _{\mu _{4},\lambda }$. By Lemma \ref%
{Le8}(iv), we see that%
\begin{equation*}
T_{0,\lambda }(\alpha )<L\leq T_{\mu _{4},\lambda }(\theta _{\mu
_{4},\lambda })=T_{\mu _{4},\lambda }(\alpha ).
\end{equation*}%
So by (\ref{Le4b}), there exists unique $\mu _{L}=\mu _{L}(\alpha )\in
\left( 0,\mu ^{\ast }\right) $ such that $T_{\mu _{L},\lambda }(\alpha )=L$.
Notice that $\mu _{4}=\bar{\mu}$ as $\alpha =\theta _{\bar{\mu},\lambda }$.
Thus $\mu _{L}(\theta _{\bar{\mu},\lambda })=\bar{\mu}$.

\smallskip Case 4. $\alpha \in (\theta _{\mu ^{\ast },\lambda },\tilde{\alpha%
}_{\mu ^{\ast },\lambda }]$. Since $T_{0,\lambda }(\alpha )<L\leq T_{\mu
^{\ast },\lambda }(\alpha )$, and by (\ref{Le4b}), there exists unique $\mu
_{L}=\mu _{L}(\alpha )\in \left( 0,\mu ^{\ast }\right) $ such that $T_{\mu
_{L},\lambda }(\alpha )=L$.

\smallskip Case 5. $\alpha \in (\tilde{\alpha}_{\mu ^{\ast },\lambda
},\gamma _{\lambda })$. By Lemmas \ref{Le0}, \ref{Le2} and \ref{Le3}, there
exists $\mu _{5}\in (0,\mu ^{\ast }]$ such that
\begin{equation}
\theta _{\mu _{5},\lambda }\leq \theta _{\mu ^{\ast },\lambda }<\tilde{\alpha%
}_{\mu ^{\ast },\lambda }<\alpha <\beta _{\mu _{5},\lambda }\text{ \ and \ }%
T_{\mu _{5},\lambda }(\alpha )>L.  \label{A1}
\end{equation}%
Since $T_{0,\lambda }(\alpha )<L$, and by (\ref{Le4b}) and (\ref{A1}), there
exists unique $\mu _{L}=\mu _{L}(\alpha )\in \left( 0,\mu ^{\ast }\right) $
such that $T_{\mu _{L},\lambda }(\alpha )=L$.

\medskip

\noindent By Cases 1--5, there exists a positive function $\mu _{L}(\alpha )$
on $[\theta _{\bar{\mu},\lambda },\gamma _{\lambda })$ such that $T_{\mu
_{L}(\alpha ),\lambda }(\alpha )=L$.

\smallskip

\textbf{Step 3.} We prove the statement (i) if $\eta <L\leq 2\eta $. If $%
g^{\prime }(0^{+})=\infty $, then the case $0=\eta <L\leq 2\eta =0$ does not
exist. Hence, we only consider the case $g^{\prime }(0^{+})\in (0,\infty )$.
It implies that $\eta >0$. By Step 1, we consider three cases.

\smallskip Case 1. $\alpha \in (0,\theta _{\mu ^{\ast },\lambda }]$. Since $%
0<\alpha \leq \theta _{\mu ^{\ast },\lambda }$, and by Lemma \ref{Le1},
there exists $\mu _{6}\in (0,\mu ^{\ast }]$ such that $\alpha =\theta _{\mu
_{6},\lambda }$. So by Step 1, Lemmas \ref{Le0} and \ref{Le8}, we see that%
\begin{equation*}
T_{\mu _{6},\lambda }(\alpha )=T_{\mu _{6},\lambda }(\theta _{\mu
_{6},\lambda })>L=T_{0,\lambda }(\gamma _{\lambda })>T_{0,\lambda }(\tilde{%
\alpha}_{\mu ^{\ast },\lambda })>T_{0,\lambda }(\theta _{\mu ^{\ast
},\lambda })\geq T_{0,\lambda }(\alpha ).
\end{equation*}%
By (\ref{Le4b}), there exists unique $\mu _{L}=\mu _{L}(\alpha )\in \left(
0,\mu _{6}\right) $ such that $T_{\mu _{L},\lambda }(\alpha )=L$. Obviously,
by Lemma \ref{Le1}, we observe that $\mu _{6}\rightarrow 0^{+}$ as $\alpha
\rightarrow 0^{+}$. So $\mu _{L}(0^{+})=0$.

\smallskip Case 2. $\alpha \in (\theta _{\mu ^{\ast },\lambda },\gamma
_{\lambda })$. The same arguments used in Cases 4 and 5 of Step 2 can be
applied to prove that there exists unique $\mu _{L}=\mu _{L}(\alpha )\in
\left( 0,\mu ^{\ast }\right) $ such that $T_{\mu _{L},\lambda }(\alpha )=L$.

\smallskip Case 3. $\alpha \in \lbrack \gamma _{\lambda },\infty )$. The
same arguments used in Case 2 of Step 2 can be applied to prove that $T_{\mu
,\lambda }(\alpha )\neq L$ for $\mu \in \left( 0,\mu _{\lambda }\right) .$

\medskip

\noindent By Cases 1--3, there exists a positive function $\mu _{L}(\alpha )$
on $(0,\gamma _{\lambda })$ such that $T_{\mu _{L}(\alpha ),\lambda }(\alpha
)=L$.

\smallskip

\textbf{Step 4.} We prove the statement (ii). By (\ref{Le4b}) and implicit
function theorem, $\mu _{L}=\mu _{L}(\alpha )$ is a continuously
differentiable function on $\mathring{I}$. Moreover,%
\begin{equation*}
0=\frac{\partial }{\partial \alpha }L=\frac{\partial }{\partial \alpha }%
T_{\mu _{L}(\alpha ),\lambda }(\alpha )=T_{\mu _{L}(\alpha ),\lambda
}^{\prime }(\alpha )+\left[ \frac{\partial }{\partial \mu }T_{\mu ,\lambda
}(\alpha )\right] _{\mu =\mu _{L}(\alpha )}\mu _{L}^{\prime }(\alpha )\text{,%
}
\end{equation*}%
which implies that (\ref{10a}) holds by (\ref{Le4b}).

\smallskip \textbf{Step 5.} We prove the statement (iii). Assume that $%
L>2\eta $. By similar argument in the proof of Lemma \ref{Le9}(iii), we
prove that $\mu _{L}(\alpha )$ is continuous on $[\theta _{\bar{\mu},\lambda
},\gamma _{\lambda })$. We omit the details. Assume that $\eta <L\leq 2\eta $%
. Since $I=\mathring{I}=(0,\gamma _{\lambda })$, and by Lemma \ref{Le10}%
(ii), $\mu _{L}(\alpha )$ is continuous on $(0,\gamma _{\lambda })$. Thus $%
\Sigma _{\lambda }=\left\{ \left( \mu _{L}(\alpha ),\alpha \right) :\alpha
\in I\right\} $ is continuous on the $\left( \lambda ,\left\Vert u_{\lambda
}\right\Vert _{\infty }\right) $-plane.

\smallskip \textbf{Step 6.} We prove the statement (iv). Let $\mu
_{7}=\limsup_{\alpha \rightarrow \gamma _{\lambda }^{-}}\mu _{L}(\alpha )$.
Since $\mu _{L}(\alpha )<\mu _{\lambda }$, we see that $\mu _{7}\in \lbrack
0,\infty )$. Suppose $\mu _{7}>0$. By Lemma \ref{Le1}, there exists $\mu
_{8}\in \left( 0,\mu _{7}\right) $ such that $\theta _{\mu _{8},\lambda
}<\gamma _{\lambda }<\beta _{\mu _{8},\lambda }$. There exists a sequence $%
\{\alpha _{n}\}\subset \left( \theta _{\mu _{8},\lambda },\gamma _{\lambda
}\right) $ such that%
\begin{equation*}
\alpha _{n}\nearrow \gamma _{\lambda }\text{ \ as }n\rightarrow \infty \text{%
, \ and \ }\mu _{L}(\alpha _{n})>\mu _{8}\text{ \ for }n\in \mathbb{N}\text{.%
}
\end{equation*}%
So by Lemma \ref{Le0} and (\ref{Le4b}), then%
\begin{equation*}
L=\lim_{n\rightarrow \infty }T_{\lambda _{L}(\alpha _{n}),\lambda }(\alpha
_{n})\geq \lim_{n\rightarrow \infty }T_{\mu _{8},\lambda }(\alpha
_{n})=T_{\mu _{8},\lambda }(\gamma _{\lambda })>T_{0,\lambda }(\gamma
_{\lambda })=L,
\end{equation*}%
which is a contradiction. Thus $\lim_{\alpha \rightarrow \gamma _{\lambda
}^{-}}\mu _{L}(\alpha )=0$. The proof is complete.

\smallskip

\textbf{Acknowledgment.} This work was supported by the National Science and
Technology Council, Taiwan, under Grant No. NSTC 113-2115-M-152-001.

\smallskip

\noindent \textbf{Declaration of generative AI and AI-assisted technologies
in the writing process}

During the preparation of this work the author used ChatGPT to improve the
clarity and readability of the manuscript. After using this tool, the author
reviewed and edited the content as needed and takes full responsibility for
the content of the published article.

\bigskip

\end{document}